\documentclass[11pt,letterpaper,reqno]{amsart}
\usepackage{mathtools}
\usepackage{amsthm}
\usepackage[T1]{fontenc}
\usepackage{enumitem} % customize enumerates
\usepackage{hyperref}\hypersetup{colorlinks=true, citecolor=blue}
\usepackage{color,todonotes}
\usepackage[mathscr]{euscript}
\usepackage{comment} 
\usepackage{amssymb}
\usepackage{amsfonts}
\usepackage{esint,amsfonts,amsmath,amssymb,epsfig,mathrsfs,bm,accents,yfonts,dsfont,stmaryrd}%MnSymbol}
\usepackage{graphicx,color}
\usepackage{tikz,fp,ifthen,fullpage,bbm}
\usetikzlibrary{backgrounds}
\usetikzlibrary{decorations.pathmorphing,backgrounds,fit,calc,through}
\usetikzlibrary{arrows}
\usetikzlibrary{shapes,decorations,shadows}
\usetikzlibrary{fadings}
\usetikzlibrary{patterns}
\usetikzlibrary{mindmap}
\usetikzlibrary{decorations.text}
\usetikzlibrary{decorations.shapes}

\theoremstyle{plain} %to change the headings as suggested by Kleiner
\newtheorem{theorem}{Theorem}[section]
\newtheorem{lem}[theorem]{Lemma}
\newtheorem{prop}[theorem]{Proposition}
\newtheorem{thm}[theorem]{Theorem}
\newtheorem{cor}[theorem]{Corollary}

\newtheorem{remark}[theorem]{Remark}

\newtheorem{notazioni}[theorem]{Notation}
\newtheorem{defn}[theorem]{Definition}
\numberwithin{equation}{section}

\newtheoremstyle{mytheorem}
{}% measure of space to leave above the theorem. E.g.: 3pt
{}% measure of space to leave below the theorem. E.g.: 3pt
{\it}% name of font to use in the body of the theorem
{\parindent}% measure of space to indent
{\bf}% name of head font
{.}% punctuation between head and body
{ }% space after theorem head; " " = normal interword space
{\thmnumber{#2.~}\thmname{#1}\thmnote{~\rm#3}}% Manually specify head

\newtheoremstyle{myremark}
{}% measure of space to leave above the theorem. E.g.: 3pt
{}% measure of space to leave below the theorem. E.g.: 3pt
{\rm}% name of font to use in the body of the theorem
{\parindent}% measure of space to indent
{\bf}% name of head font
{.}% punctuation between head and body
{ }% space after theorem head; " " = normal interword space
{\thmnumber{#2.~}\thmname{#1}\thmnote{~\rm#3}}% Manually specify head

\newtheoremstyle{myparagraph}
{}% measure of space to leave above the theorem. E.g.: 3pt
{}% measure of space to leave below the theorem. E.g.: 3pt
{\rm}% name of font to use in the body of the theorem
{\parindent}% measure of space to indent
{\bf}% name of head font
{.}% punctuation between head and body
{ }% space after theorem head; " " = normal interword space
{\thmnumber{#2.~}\thmname{#1}\thmnote{#3}}% Manually specify head

\makeatletter
\def\@secnumfont{\sc}
\def\section{\@startsection{section}{1}%
\z@{1.5\linespacing\@plus .2\linespacing}{.7\linespacing}%
{\normalfont\sc\centering}}
\makeatother
\makeatletter
\def\ps@headings{\ps@empty
 \def\@evenhead{%
  \setTrue{runhead}%
  \normalfont\footnotesize
  \rlap{\thepage}\hfil
  \def\thanks{\protect\thanks@warning}%
  \leftmark{}{}\hfil}%
 \def\@oddhead{%
  \setTrue{runhead}%
  \normalfont\footnotesize\hfil
  \def\thanks{\protect\thanks@warning}%
  \rightmark{}{}\hfil \llap{\thepage}}%
\let\@mkboth\markboth}
\makeatother
\makeatletter
\renewenvironment{proof}[1][\proofname]{\par
  \pushQED{\qed}%
  \normalfont \topsep6\p@\@plus6\p@\relax
  \trivlist
  \itemindent\normalparindent
  \item[\hskip\labelsep
    \bfseries
    #1\@addpunct{.}]\ignorespaces
}{%
  \popQED\endtrivlist\@endpefalse
}
\providecommand{\proofname}{Proof}
\makeatother
\newcommand{\Mass}{\mathbb{M}}

\newcommand{\R}{\mathbb{R}}
\newcommand{\Q}{\mathbb{Q}}
\newcommand{\N}{\mathbb{N}}

\newcommand{\F}{\mathscr{F}}
\newcommand{\G}{\mathrm{G}}
\newcommand{\Haus}{\mathcal{H}}

\newcommand{\M}{\mathbb{M}}

\newcommand{\Lip}{\mathrm{Lip}}

\newcommand{\Tan}{\mathrm{Tan}}

\newcommand{\dist}{\mathrm{dist}}

\newcommand{\dV}{d_V\kern-1pt}
\newcommand{\cost}{{\mathcal C}}
\newcommand{\en}{\mathbb E}

\DeclareMathOperator{\sign}{sign}
\DeclareMathOperator{\rec}{rec}
\DeclareMathOperator{\diff}{diff}
\DeclareMathOperator{\spn}{span}

\DeclareMathOperator{\spt}{spt}
\newcommand{\trait}[3]{\vrule width #1ex height #2ex depth #3ex}
\newcommand{\trace}{\mathchoice%
  {\mathbin{\trait{.12}{1.2}{.03}\trait{.8}{0.09}{0.03}}}
  {\mathbin{\trait{.12}{1.2}{.03}\trait{.8}{0.09}{0.03}}}
  {\mathbin{\hskip.15ex\trait{.09}{.84}{0.02}\trait{.56}{.07}{.02}}\hskip.15ex}
  {\mathbin{\trait{.07}{.6}{.01}\trait{.4}{.06}{.01}}}}
\makeatletter
\@addtoreset{footnote}{section}
\makeatother
	\newcommand{\wto}{\stackrel{*}{\rightharpoonup}}

\newcommand{\Na}{\mathbb{N}} %naturali
\newcommand{\Sf}{\mathbb{S}} %sphere

\newcommand{\p}{\mathbf{p}} %projection
\newcommand{\mass}{\mathbb{M}} %mass of a current
\newcommand{\Po}{\mathscr{P}} %polyhedral chains
\newcommand{\Rc}{\mathscr{R}} %integer rectifiable currents
\renewcommand{\flat}{\mathbb{F}} % integral or real flat distance
\newcommand{\Ha}{\mathcal{H}} %Hausdorff distance

\newcommand{\e}{\varepsilon} %epsilon
\newcommand{\diam}{\mathrm{diam}} %diameter
\def\XXint#1#2#3{{\setbox0=\hbox{$#1{#2#3}{\int}$ }
\vcenter{\hbox{$#2#3$ }}\kern-.6\wd0}}

\newcommand{\mres}{\mathbin{\vrule height 1.6ex depth 0pt width %measure restriction
0.13ex\vrule height 0.13ex depth 0pt width 1.3ex}}
\DeclareFontFamily{U}{matha}{\hyphenchar\font45}
\DeclareFontShape{U}{matha}{m}{n}{
  <-6> matha5 <6-7> matha6 <7-8> matha7
  <8-9> matha8 <9-10> matha9
  <10-12> matha10 <12-> matha12
  }{}
\DeclareSymbolFont{matha}{U}{matha}{m}{n}
\DeclareMathSymbol{\Lt}{3}{matha}{"CE}

\newcommand{\abs}[1]{\lvert#1\rvert} %absolute value
\newcommand{\E}{\mathbb{E}}

\title{A multi-material transport problem\\ with arbitrary marginals}

\author[Marchese]{A. Marchese}
\address[Andrea Marchese]{
\newline \indent Dipartimento di Matematica, Universit\`a degli Studi di Trento,
\newline \indent Via Sommarive 14, 38123 Povo, Italy}
\email{andrea.marchese@unitn.it}

\author[Massaccesi]{A. Massaccesi}
\address[Annalisa Massaccesi]{
\newline \indent Dipartimento di Tecnica e Gestione dei Sistemi Industriali, Università di Padova,
\newline \indent Stradella San Nicola, 3 - Vicenza, Italy}
\email{annalisa.massaccesi@unipd.it}

\author[Stuvard]{S. Stuvard}
\address[Salvatore Stuvard]{
\newline \indent Department of Mathematics, University of Texas at Austin,
\newline \indent 2515 Speedway STOP C1200, 78712 Austin, Texas, USA}
\email{stuvard@math.utexas.edu}

\author[Tione]{R. Tione}
\address[Riccardo Tione]{
\newline \indent EPFL B, Station 8,
\newline \indent CH-1015 Lausanne, Switzerland}
\email{riccardo.tione@epfl.ch}

\begin{document}
\begin{abstract}
In this paper we study general transportation problems in $\mathbb{R}^n$, in which $m$ different goods are moved simultaneously. The initial and final positions of the goods are prescribed by measures $\mu^-$, $\mu^+$ on $\mathbb{R}^n$ with values in $\mathbb{R}^m$. When the measures are finite atomic, a discrete transportation network is a measure $T$ on $\mathbb{R}^n$ with values in $\mathbb{R}^{n\times m}$ represented by an oriented graph $\mathcal{G}$ in $\mathbb{R}^n$ whose edges carry multiplicities in $\mathbb{R}^m$. The constraint is encoded in the relation ${\rm div}(T)=\mu^--\mu^+$. The cost of the discrete transportation $T$ is obtained integrating on $\mathcal{G}$ a general function $\mathcal{C}:\mathbb{R}^m\to\mathbb{R}$ of the multiplicity. When the initial data $\left(\mu^-,\mu^+\right)$ are arbitrary (possibly diffuse) measures, the cost of a transportation network between them is computed by relaxation of the functional on graphs mentioned above. Our main result establishes the existence of cost-minimizing transportation networks for arbitrary data $\left(\mu^-,\mu^+\right)$. Furthermore, under additional assumptions on the cost integrand $\mathcal{C}$, we prove the existence of transportation networks with finite cost and the stability of the minimizers with respect to variations of the given data. Finally, we provide an explicit integral representation formula for the cost of rectifiable transportation networks, and we characterize the costs such that every transportation network with finite cost is rectifiable.
\end{abstract}

\maketitle

\medskip\noindent
{\sc Keywords: } Transportation networks, Branched transportation, Multi-material transport problem, Normal currents, Currents with coefficients in groups.

\par
\medskip\noindent
{\sc MSC :} 49Q10, 49Q15, 49Q20.
\par

\tableofcontents
\section{Introduction}
In several transportation problems one may be interested in the minimization of a cost functional which privileges the aggregation of mass particles during the transportation and prevents diffusion. This automatically produces optimal transportation networks with branched structures. The branching behavior of optimal transportation systems emerges in many natural phenomena, such as the structure of the nerves of a leaf and the roots of a tree, of river basins and of the bronchial, the cardiovascular, and the nervous system, as well as in several human-designed supply-demand systems, like water and energy distribution or urban planning.\\

The most popular \emph{Eulerian formulation} of branched transportation was proposed by Xia in \cite{Xia2003}: in this model, a 1-rectifiable vector-valued measure on the underlying ambient space $\R^n$ (also called a 1-dimensional rectifiable current) $T=\vec T\|T\|$ is regarded as a transportation network connecting an initial positive measure $\mu^-$ to a target positive measure $\mu^+$ with the same mass. Here, $\|T\|$ is a positive Radon measure which is absolutely continuous with respect to the Hausdorff measure $\Haus^1$ restricted to a $1$-rectifiable set $E \subset \R^n$, and $\vec T$ is a unit vector field on $\R^n$, which is tangent to $E$ at $\|T\|$-almost every point. The condition that $T$ transports $\mu^-$ onto $\mu^+$ is encoded in the relation ${\rm div}\, T=\mu^--\mu^+$, which generalizes the classical Kirchhoff circuit laws. At every point $x$ in the ambient space, the direction of the flow of mass through $x$ and the intensity of the flow are represented respectively by $\vec T(x)$ and by the Radon-Nikod\'ym density $\theta(x)$ of $\|T\|$ with respect to the measure $\Haus^1$ restricted to $E$. The cost of such a transportation network is obtained integrating on $E$ a fractional power $\alpha\in(0,1)$ of the density $\theta$. In \cite{Maddalena2003}, an equivalent model was proposed by Maddalena, Morel, and Solimini, who presented a \emph{Lagrangian formulation} of the problem, in which one traces the trajectory of each mass particle, thus gaining the possibility to introduce stricter types of constraint (see the description of the \emph{mailing problem} in \cite{Bernot2009}). The equivalence of the Eulerian and Lagrangian formulations has recently been extended to models corresponding to different cost functionals, see e.g. \cite{Brancolini2016}. \\

The model introduced above describes the transportation of a single material. In the present paper, we are interested in the possibility to transport \emph{simultaneously} a number $m$ of different types of goods or commodities. We want to allow the interaction between different commodities to be independent: for instance, aggregating two unit masses of a certain pair of commodities might be more or less convenient than aggregating two unit masses of a different pair. In particular, the cost per unit length of the transportation of a collection of goods will depend not only on the total mass of that collection, but on the actual array whose components represent the masses (and the directions) of each single commodity. An example which justifies our interest is given by the \emph{power line communication technology} (PLC), which uses the electric power distribution network for data transmission; see \cite{MMT} and references therein. Even though electricity and data signal can be transported along the same network, they should not be treated as a single material, for example because the users' concentration and demands are not necessarily proportional. \\
 
In analogy with the model proposed by Xia, our given datum is an $m$-tuple of initial (positive) measures $(\mu_1^-,\dots,\mu_m^-)$ on $\R^n$, and an $m$-tuple of target (positive) measures $(\mu_1^+,\dots,\mu_m^+)$. For $j=1,\dots,m$, $\mu_j^-$ and $\mu_j^+$ represent the initial and the target distribution of the $j$-th commodity respectively (and therefore they must have equal masses). The difference between the initial and the target $m$-tuples can be written as a vector-valued measure $\nu$ on $\R^n$ with values in $\R^m$. The equality between the masses of each component is rephrased by requiring that $\nu(v)=0$ for every constant vector field $v:\R^n\to\R^m$. A transportation network connecting the initial measures to the target ones is a vector-valued measure $T=\vec T\|T\|$, where $\|T\|$ is a positive Radon measure on $\R^n$ and $\vec T:\R^n\to\R^{n\times m}$ is a unit vector field. The constraint is given by the relation ${\rm div}\, T=\nu$. In the language of Geometric Measure Theory, such objects are 1-dimensional normal currents in $\R^n$ \emph{with coefficients in $\R^m$}, and the divergence constraint corresponds to prescribing the boundary of the currents.\\

The cost of a multi-material transportation network is defined as follows. We consider a rather general function $\cost:\R^m\to[0,\infty)$, and we use it to first define the cost functional on a special class of measures, those which in the language of currents are called 1-dimensional \emph{polyhedral currents} in $\R^n$ with coefficients in $\R^m$. These are $\R^{n \times m}$-valued measures on $\R^n$ of the form 
\[
T = \sum_{{\bf e} \in E(\mathcal{G})} (\tau_{\bf e} \otimes \theta({\bf e})) \, \Ha^1 \trace {\bf e}\,,
\]
thus supported on a finite union of non-overlapping segments ${\bf e}$ (the \emph{edges} of a finite graph $\mathcal G$) oriented by $\tau_{{\bf e}} \in \R^n$ and with multiplicities $\theta({\bf e}) \in \R^m$; see the formal definition in Subsection \ref{multi_material_fluxes}. The cost functional, or \emph{energy} $\E$, for such a $T$ is simply obtained integrating on the finite graph $\mathcal{G}$ the function $\cost(\theta_1({\bf{e}}),\dots,\theta_m({\bf{e}}))$ with respect to the measure $\Haus^1$, namely
$$\E(T):=\sum_{{\bf e}\in E(\mathcal{G})}\cost(\theta_1({\bf{e}}),\dots,\theta_m({\bf{e}}))\Haus^1({\bf{e}}),$$
which is well-defined and lower semi-continuous on the class of polyhedral currents, under minimal assumptions on the function $\cost$ (see Definition \ref{def:mmtcost}). Heuristically, $\cost(\theta_1,\dots,\theta_m)$ represents the cost per unit length for the joint transportation of an amount $\theta_1$ of the commodity indexed by $1$, together with an amount $\theta_2$ of the commodity indexed by $2$, etc... Different signs of the $\theta_j$'s encode the possibility to transport the corresponding commodities with two possible orientations along each stretch of the transportation network. The energy of a general transportation network $T$ is defined via relaxation.\\

The model of multi-material transport presented above is the natural extension of the discrete model proposed in \cite{MMT}. We remark that the possibility to describe the transportation of a vector-valued quantity via 1-dimensional currents with coefficients in $\R^m$ was also suggested in the final comments of \cite{Brancolini2017}. In the present work, we show the existence of transportation networks minimizing the energy $\E$ for any admissible choice of the source and target measures, under minimal assumptions on the cost function $\cost$. Under mild additional assumptions, we prove that for any admissible given data the corresponding minimizers have finite energy. Under the same assumptions, we also show that the multi-material transportation problem is stable: namely, a sequence of minimizers corresponding to a converging family of given data converges to a minimizer of the limit problem. Finally, we provide an explicit representation formula for the energy of a class of transportation networks exhibiting a nice geometric structure, more precisely for those networks which can be described by \emph{rectifiable} $1$-currents. This is an important class of solutions: indeed, under very natural assumptions on $\cost$ all transportation networks with finite energy are necessarily rectifiable. In these cases, one expects the solutions to exhibit fractal-type behaviors, as it happens in the single-material setting; see e.g. \cite{Brancolini2014,PegonSantambrogioXia}.\\

The rest of the paper will be divided in $8$ sections. After briefly introducing some basic notation in Geometric Measure Theory in Section \ref{s:notation}, Section \ref{s:mmtp} will contain a formal introduction to the multi-material transportation problem together with the statement of our main results. For the sake of simplicity, we will adopt the language of vector-valued measures in order to obtain such first phrasing of the problem, even though the formalism which is best tailored to describe multi-material optimal transport is that offered by the \emph{theory of normal currents with coefficients in $\R^m$}. Indeed, as it will become apparent in the sequel, standard tools in the theory of currents allow a natural description of the geometric operations on transportation networks that are essential in our proofs; see e.g. the role of the slicing operators in the proof of our existence Theorem \ref{t:existence}, in particular Proposition \ref{prop:monotonicity}. For these reasons, Section \ref{s:syllabus_euclidean} is fully devoted to the development of the formalism of normal currents with coefficients in $\R^m$ and to the rephrasing of the multi-material transport problem within such framework. After a thorough analysis of the basic properties of the energy functional carried out in Section \ref{s:prelim}, we prove the general existence theorem of minimizers in Section \ref{s:existence}. In Section \ref{s:finite_energy} we discuss, instead, the existence of transportation networks with finite energy and the stability of the problem with respect to perturbations of the source and target measures. The representation formula for the energy of rectifiable networks concludes the paper. It is proved in Section \ref{s:representation} in the general setting of rectifiable chains of arbitrary dimension $k$ and with coefficients in a normed Abelian group ${\rm G}$, of which rectifiable $k$-currents with coefficients in $\R^m$ are a special case: the relevant preliminaries on ${\rm G}$-chains are collected in Section \ref{currents} for the reader's convenience.

\medskip 
\noindent\textbf{Acknowledgments.} This project has received funding from the European Union's Horizon 2020 research and innovation programme under the grant agreement No. 752018 (CuMiN). The first and second named authors have benefited from partial support from INdAM-GNAMPA. The third named author has been supported by the NSF Grants DMS-1565354, DMS-1361122 and DMS-1262411. The fourth named author has been partially supported by the SNF Grant 182565.

\section{Basic notation}\label{s:notation}

We shall always work in Euclidean space $\R^n$ with $n \geq 2$. The standard orthonormal basis of $\R^n$ is denoted by $(e_1,\dots,e_n)$, and the coordinates of a vector $a\in\R^n$ with respect to this basis are $(a_1,\dots,a_n)$. Given $a\in\R^n$ and $b\in\R^m$, we denote by $a\otimes b$ the element of $\R^{n\times m}\sim {\rm Mat}(n\times m)$ defined by $(a\otimes b)_{ij}:=a_ib_j$ for $i \in \{1,\ldots,n\}$ and $j \in \{1,\ldots,m\}$. We will make use of standard notation in multilinear algebra. In particular, the vector spaces of $k$-vectors and $k$-covectors in $\R^n$ ($1\leq k\leq n$) are denoted, respectively, by $\Lambda_k(\R^n)$ and $\Lambda^k(\R^n)$. We shall regard $\Lambda_{k}(\R^n)$ and $\Lambda^{k}(\R^n)$ as normed vector spaces with the \emph{mass norm} $|\cdot|$ and \emph{comass norm} $\|\cdot\|$ respectively (see \cite[1.8.1]{Federer_GMT}). We write $B(x,r)$ for the open ball of center $x\in \R^n$ and radius $r>0$. The symbol $|\cdot|$ will always denote the Euclidean norm in $\R^n$, and we will set $\Sf^{n-1} := \{ x \in \R^n \, \colon \abs{x} = 1 \}$. The characteristic function of a set $A$, taking values 0 and 1, is denoted by $\mathbbm 1_{A}$.
We denote by $\mathscr{M}(\R^n)$ the space of signed Radon measures on $\R^n$, namely the vector space of real-valued measures $\mu$ on the $\sigma$-algebra of Borel sets whose \emph{negative and positive parts}
$$\mu_-:=\frac{\|\mu\|-\mu}{2} \quad\mbox{and}\quad \mu_+ := \frac{\|\mu\|+\mu}{2}.$$
are Radon measures. Here, as usual, $\| \mu\|$ denotes the total variation of $\mu \in \mathscr{M}(\R^n)$. We denote also by $\mathscr{M}_+(\R^n)$ the subset of positive measures. Given a normed vector space $V$ with dual $V^*$, the duality pairing between two elements $w\in V^*$ and $v\in V$ is denoted by $\langle w,v\rangle$. We denote by $\mathscr{M}(\R^n,V)$ the space of vector-valued measures with values in $V$. By the Radon-Nikod\'ym theorem, every measure $T\in\mathscr{M}(\R^n,V)$ can be uniquely written as
\begin{equation}\label{rep_meas}
T=\vec T \|T\|\,,
\end{equation}
where $\|T\|\in\mathscr{M}_+(\R^n)$ is the \emph{total variation measure} of $T$ and $\vec T:\R^n\to V$ is a unit vector field, in the sense that $\|\vec T (x) \|_V = 1$ for $\|T\|$-a.e. $x \in \R^n$. The equality \eqref{rep_meas} means that, for every continuous vector field $w:\R^n\to V^*$ with compact support, it holds
$$T(w)=\int_{\R^n}\langle w,\vec T\rangle\;d\|T\|.$$
The \emph{mass} of a measure $T \in\mathscr{M}(\R^n,V)$ is the quantity 
$$\mass(T):=\|T\|(\R^n).$$

We denote by 
$$\spt (\mu):= \bigcap\{C\subset \R^n:C \mbox{ is closed and } \|\mu\|(\R^n \setminus C)=0\}$$
the \emph{support} of $\mu$. We say that $\mu$ is \emph{supported} on a Borel set $E$ if $\|\mu\|(\R^n\setminus E)=0$. We say that $\mu$ is \emph{atomic} if it is supported on a countable set, and \emph{discrete} or \emph{finite atomic} if it is supported on a set of finitely many points.
If $\mu$ is a Radon measure in $\R^n$ and $f \in L^1_{loc}(\R^n, \left(0,\infty\right);\mu)$ then we let $f\,\mu$ denote the Radon measure \[ (f\,\mu)(E) := \int_E f \, d\mu\,.\] In particular, for a measure $\mu\in\mathscr{M}(\R^n)$ and a Borel set \(E\subset \R^n\),  \(\mu\trace E\) is the \emph{restriction} of \(\mu\) to \(E\), i.e. the measure $\mathbbm 1_{E}\,\mu$.
We say that two measures $\mu$ and $\nu$ are mutually \emph{singular} if there exists a Borel set $E$ such that $\|\mu\|=\|\mu\|\trace E$ and $\|\nu\|=\|\nu\|\trace E^c$, where $E^c := \R^n \setminus E$. If $\{\mu_h\}_{h=1}^\infty$ is a sequence in $\mathscr{M}(\R^n,V)$, we say that $\mu_h$ weakly-* converges to $\mu \in \mathscr{M}(\R^n,V)$, and we write $\mu_h \wto \mu$, if  
\[
\lim_{h \to \infty} \mu_h(w) = \mu(w) \qquad \mbox{for every $w \in C_c(\R^n,V^*)$}\,.
\]

We use $\Haus^k$ to denote the $k$-dimensional Hausdorff measure, see \cite{Simon1983}. A $\Haus^k$-measurable set $E \subset \R^n$ is (countably) $k$-rectifiable if it can be covered by countably many $k$-dimensional Lipschitz graphs up to a $\Haus^k$-negligible set, that is if there are countably many Lipschitz functions $f_h \colon \R^k \to \R^n$ such that
\[
\Haus^k\left( E \setminus \bigcup_{h=1}^\infty f_h(\R^k) \right) = 0\,.
\]
A fundamental property of a $k$-rectifiable set $E$ in $\R^n$ is the existence of an \emph{approximate tangent plane} ${\rm Tan}(E,x)$ at $\Ha^k$-a.e. $x \in E$. This is a unique $k$-plane $\Pi$ in $\R^n$ characterized by the following property: there exists a function $g \in L^1_{loc}(E, \left( 0, \infty \right) ; \Ha^k)$ such that \[g(x + r\,\cdot) \, \Ha^k \trace \left( \frac{E-x}{r} \right) \wto g(x)\, \Ha^k \trace \Pi\] as Radon measures in $\R^n$ when $r \to 0^+$; see \cite[Definition 11.4]{Simon1983}.

\section{Multi-material transport problem}\label{s:mmtp}
A \emph{discrete} model for the multi-material transport problem is described in \cite{MMT}, using 1-dimensional integral currents with coefficients in  $\R^m$ ($m$ being the number of transported commodities). In that paper, the particles are assumed to have integer-valued masses (or, equivalently, integer multiples of a fixed real number). Here we describe a \emph{continuous} model, obtained via relaxation of a cost functional (similar to that introduced by Gilbert \cite{Gilbert1967}) defined on discrete transportation networks represented by directed graphs with multiplicities in $\R^m$. Although a proper description of the model would require notions from the theory of currents with coefficients in groups, in this section we present the model and we state the main results of the paper using the language of vector-valued measures, in order to make the content of the paper more accessible also to readers who are not familiar with the theory of currents. A drawback of this simplified presentation is the fact that, in the definition of cost functional, we need to use a notion of convergence (called flat-convergence) which is defined for currents and it would not have a natural definition for vector-valued measures. Hence, we will postpone the definition of such convergence to Section \ref{s:syllabus_euclidean}, where we present a brief summary of the notions from the theory of currents with coefficients in $\R^m$ that are used throughout the paper. 

\subsection{Multi-material fluxes}\label{multi_material_fluxes}
A 1-dimensional \emph{polyhedral current} in $\R^n$ with coefficients in $\R^m$ is a matrix-valued measure $T \in\mathscr{M}(\R^n,\R^{n\times m})$ of the form
\[
T = \sum_{{\bf e} \in E(\mathcal{G})} (\tau_{\bf e} \otimes \theta({\bf e})) \, \Ha^1 \trace {\bf e}\,,
\]
where:
\begin{itemize}
\item [(i)] $\mathcal{G}\subset \R^n$ is a finite graph, i.e. a set consisting of a finite union of closed line segments. The collection of all such segments is denoted $E(\mathcal{G})$, and each element ${\bf e}\in E(\mathcal{G})$ is called an \emph{edge} of the graph $\mathcal{G}$. We will assume that the edges are non-overlapping, i.e. two edges may intersect only at the end-points;
\item [(ii)] for each edge ${\bf e}\in E(\mathcal{G})$, $\tau_{\bf e} \in \Sf^{n-1}$ is a fixed orientation of ${\bf e}$, and ${\bf \theta}({\bf e}):=(\theta_1({\bf e}),\dots,\theta_m({\bf e}))\in\R^m$. Thus, $\tau_{\bf e} \otimes \theta({\bf e})$ is a rank-1 $(n \times m)$-matrix with all columns parallel to ${\bf e}$. We will call $\theta({\bf e})$ the vector-valued \emph{multiplicity} associated to ${\bf e}$ (note that $\theta({\bf e})$ is defined up to a sign, given that both $\tau_{\bf e}$ and $- \tau_{\bf e}$ are suitable orientations for ${\bf e}$).
\end{itemize}
Let us call $x_{\bf e}$ and $y_{\bf e}$ the end-points of ${\bf e}$, with the convention that $y_{\bf e}-x_{\bf e}$ is a positive multiple of $\tau_{\bf e}$. It is easy to check that the distributional divergence of $T$, namely the $\R^m$-valued distribution defined by
\[
{\rm div}\, T (\phi) := - T (D\phi) \qquad \mbox{for every $\phi \in C^\infty_c(\R^n,(\R^m)^*)$}\,,
\]
(with the obvious identifications) satisfies
$${\rm div}\,T=\sum_{e\in E(\mathcal{G})}{\bf \theta}({\bf e})(\delta_{x_{\bf e}}-\delta_{y_{\bf e}}),$$
where we denoted with $\delta_P$ the Dirac mass at the point $P\in\R^n$. The latter observation motivates the following definition.

\begin{defn}[Discrete multi-material flux]
Given two discrete vector-valued measures $\mu^-,\mu^+\in\mathscr{M}(\R^n,\R^m)$, and given $T\in\mathscr{M}(\R^n,\R^{n\times m})$ a 1-dimensional polyhedral current in $\R^n$ with coefficients in $\R^m$, we say that $T$ is a \emph{discrete multi-material flux between $\mu^-$ and $\mu^+$} if ${\rm div}\,T=\mu^--\mu^+$.
\end{defn}

Observe that a necessary condition for the existence of a discrete multi-material flux between two discrete vector-valued measures $\mu^-$ and $\mu^+$ is that $\mu^-(v)=\mu^+(v)$ for every constant vector field $v:\R^n\to\R^m$. The condition is also sufficient: indeed, given $\mu^-:=\sum_{\ell=1}^L\theta^-_\ell\delta_{x_\ell}$ and $\mu^+:=\sum_{h=1}^H\theta^+_h\delta_{y_h}$, the \emph{cone} $T$ over $\mu^+-\mu^-$ with vertex $0$ satisfies ${\rm div}\,T=\mu^--\mu^+$. This is defined as 
\begin{equation}\label{e:cone}
T:=\sum_{h=1}^H(\tau^+_h\otimes\theta^+_h)\Haus^1\trace S^+_h-\sum_{\ell=1}^L(\tau^-_\ell\otimes\theta^-_\ell)\Haus^1\trace S^-_\ell\,,
\end{equation} 
where we denoted $\tau^+_h$ and $\tau^-_\ell$ the unit vectors obtained normalizing $y_h$
and $x_\ell$ respectively (or $0$ if the corresponding point is the origin) and by $S^+_h$ and $S^-_\ell$ the segments joining $y_h$ and $x_\ell$ to the origin. 
The general definition of a multi-material flux between two (possibly diffuse) measures involves ``general'' matrix-valued measures.

\begin{defn}[Multi-material flux]\label{def_multi_mat_flux}
Given two vector-valued measures $\mu^-,\mu^+\in\mathscr{M}(\R^n,\R^m)$ with compact support, a matrix-valued measure $T\in\mathscr{M}(\R^n,\R^{n\times m})$ is a \emph{multi-material flux between $\mu^-$ and $\mu^+$} if its support is compact and ${\rm div}\,T = \mu^- - \mu^+$.
\end{defn}

Again, a necessary and sufficient condition for the existence of a multi-material flux between two compactly supported vector-valued measures $\mu^-$ and $\mu^+$ is that $\mu^-(v)=\mu^+(v)$ for every constant vector field $v:\R^n\to\R^m$. In this case, we say that the vector-valued measures $\mu^-$ and $\mu^+$  are \emph{compatible}. To check that the condition is sufficient one should generalize the argument given for the discrete setting, via the so called \emph{cone construction} (see \cite[4.3.14]{Federer_GMT}).

\begin{remark}[Normal currents with coefficients in $\R^m$ and multi-material fluxes]
In the language of currents (which we introduce in Section \ref{s:syllabus_euclidean}), every compactly supported one-dimensional normal current $T$ in $\R^n$ with coefficients in $\R^m$ having boundary $\mu^+-\mu^-$ is a multi-material flux between $\mu^-$ and $\mu^+$. The non-emptiness of the class of competitors is guaranteed again by the cone construction.
\end{remark}

\begin{remark}[Multi-material fluxes as transportation networks]\label{interpr}
Let $\nu=\vec\nu\|\nu\|$ be the difference $\mu^--\mu^+$. Writing $\vec\nu$ in components with respect to the standard basis of $\R^m$, one can represent $\nu$ via an $m$-tuple of real-valued measures $\nu_j$ ($j=1,\dots,m$) (the \emph{components} of $\nu$), where, for $j=1,\dots,m$, we denoted
$$\nu_j(A):=\nu(e_j\mathbbm 1_{A}) \quad\mbox{for every Borel set $A\subset\R^n$}.$$
Similarly, a multi-material flux $T$ between $\mu^-$ and $\mu^+$ can be represented via an $m$-tuple of vector-valued measures $T_j\in\mathscr{M}(\R^n,\R^n)$ (the components of $T$) by
$$T_j(v):=T(v\otimes e_j) \quad\mbox{for every Borel vector field $v:\R^n\to\R^n$}.$$
Denoting, for $j=1,\dots,m$, $(\nu_j)_-$ and $(\nu_j)_+$ the negative and the positive part of the real-valued measure $\nu_j$ respectively, the vector-valued measures $T_j$ are ``classical'' mass-fluxes between the measures $(\nu_j)_-$ and $(\nu_j)_+$ as in \cite[Definition 2.1]{Brancolini2017}. In conclusion, the multi-material flux $T$ can be interpreted as a transportation network which moves simultaneously the mass $(\nu_j)_-$ of the commodity indexed by $j$ onto the mass $(\nu_j)_+$, for every $j=1,\dots,m$.
\end{remark}

\subsection{The cost functional}
Generalizing \cite{Xia2003} (see also \cite{Brancolini2017}), we define a general \emph{multi-material transportation cost} $\cost:\R^m\to[0,\infty)$, and we define the cost functional (also called energy and therefore denoted $\E$) of a discrete multi-material flux $T$ associated to a finite graph $\mathcal{G}$ with multiplicity $\theta$ in $\R^m$, integrating $\cost({\bf \theta})$ on $\mathcal{G}$ with respect to $\Haus^1$. The cost functional of a general multi-material flux is defined via relaxation. \\

We first define a partial order $\preceq$ on $\R^m$ as follows: we write ${\bf \eta}\preceq{\bf \theta}$ if and only if $\sign ({\bf \eta}_j) \sign ({\bf \theta}_j)\geq 0$ and $|{\bf \eta}_j|\leq |{\bf \theta}_j|$, for every $j\in 1,\dots,m$. Notice that points which belong to the interior of distinct orthants are not comparable.  
\begin{defn}[Multi-material transportation cost]\label{def:mmtcost}
A \emph{multi-material transportation cost} is a function $\cost:\R^m\to[0,\infty)$ such that
\begin{itemize}
\item[(i)] $\cost$ is even and $\cost({\bf \theta})=0$ if and only if ${\bf \theta}=0$;
\item[(ii)] $\cost$ is lower semi-continuous;
\item[(iii)] $\cost$ is subadditive, i.e. $\cost({\bf \eta}+{\bf \theta})\leq \cost({\bf \eta})+\cost({\bf \theta})$;
\item[(iv)] $\cost$ is monotone non-decreasing, i.e. $\cost({\bf \eta})\leq \cost({\bf \theta})$ if ${\bf \eta}\preceq {\bf \theta}$.
\end{itemize}
\end{defn}

\begin{remark}
It is worth noticing that the conditions (i) to (iv) on $\cost$ are natural assumptions in the problem we want to describe. The transportation cost is non-negative, it vanishes only when there is no mass to transport, and it does not depend on the orientation of the net flow of each commodity, which justifies (i); without (iii), it would be easy to produce counterexamples to the existence of solutions; furthermore, the validity of (iii) with a strict inequality (whenever $\eta$ and $\theta$ are non-zero) produces \emph{branched} solutions; as a consequence of (iv), the cost does not decrease if the net flow of each single commodity does not decrease, as one would expect; finally, (ii) is necessary to ensure that the relaxed functional $\E$ induced by $\cost$ on general multi-material fluxes coincides with the original one on discrete multi-material fluxes as described in Definition \ref{def_cost_functional} below.
\end{remark}

\begin{defn}[Cost functional]\label{def_cost_functional}
\quad
\begin{itemize}
\item [(i)](Discrete case) Given a discrete multi-material flux $T$ associated to a finite graph $\mathcal{G}$ with multiplicity $\theta$ in $\R^m$, its \emph{cost functional} (or \emph{energy}) is the quantity
$$\E(T):=\sum_{{\bf e}\in E(\mathcal{G})}\cost({\bf \theta}({\bf e}))\Haus^1({\bf e}).$$
\item [(ii)](General case) Given two compactly supported, compatible vector-valued measures $\mu^-,\mu^+\in\mathscr{M}(\R^n,\R^m)$ and given $T\in\mathscr{M}(\R^n,\R^{n\times m})$ a multi-material flux between $\mu^-$ and $\mu^+$, we define
$$\E(T):=\inf\{\liminf_h \E(T_h)\,:\, \flat(T_h-T)\to 0\},$$
where $T_h$ are discrete multi-material fluxes between discrete measures $\mu^-_h$ and $\mu^+_h$, all supported on a common compact set, and $\flat$ denotes the flat-distance between the associated flat currents (see \S \ref{ss:flat}). 
\end{itemize}
\end{defn}
\begin{remark}[Comments on the definition]
\quad
\begin{itemize}
\item [(i)] (Discrete case) Observe that the energy is well-defined: in particular, since $\cost$ is even, $\E$ does not depend on the orientation chosen on each edge ${\bf e}\in E(\mathcal{G})$. 
\item [(ii)] (General case) We will give a precise definition of flat-distance later. For the moment, we can anticipate that, whenever $$\sup_h\{\mass(T_h)+\mass(\mu^-_h-\mu^+_h)\}<\infty,$$ it holds 
$$\flat(T_h-T)\to 0 \iff (T_h\wto T \mbox{ and } (\mu^-_h-\mu^+_h)\wto (\mu^--\mu^+)).$$
Nevertheless, we remark that the condition $\flat(T_h-T)\to 0$ does not imply in general that the masses of the $T_h$'s and of the $(\mu^-_h-\mu^+_h)$'s are equi-bounded.

The existence of discrete multi-material fluxes $T_h$ between discrete measures $\mu_h^-$ and $\mu_h^+$ with $\flat(T_h - T) \to 0$ (and thus with $\flat((\mu_h^- - \mu_h^+) - (\mu^- - \mu^+)) \to 0$) is a consequence of the polyhedral approximation theorem for normal currents; see \cite[Theorem 4.2.24]{Federer_GMT}. 

\end{itemize}
\end{remark}
\subsection{Statement of the problem and main existence result}
Now we can naturally define the following minimization problem.
\begin{defn}[Multi-material transport problem] \label{defn:mmtprob}
Given a pair of compactly supported, compatible vector-valued measures $\mu^-$ and $\mu^+\in\mathscr{M}(\R^n,\R^m)$, we say that a multi-material flux $T\in\mathscr{M}(\R^n,\R^{n\times m})$ between $\mu^-$ and $\mu^+$ is a solution of the multi-material transport problem for the pair $(\mu^-,\mu^+)$ if 
$$\E(T)\leq \E(S), \quad\mbox{ for every multi-material flux $S$ between $\mu^-$ and $\mu^+$}.$$ 
\end{defn}

We will prove the following result.
\begin{theorem}[Existence of minimizers]\label{t:existence}
Let $\mu^-$ and $\mu^+\in\mathscr{M}(\R^n,\R^m)$ be a pair of compactly supported, compatible vector-valued measures. Then the associated multi-material transport problem admits a solution.
\end{theorem}

\subsection{Stability of minimizers}
Once the existence of solutions has been guaranteed, it is natural to ask whether minimizers of multi-material transport problems enjoy a stability property, that is, whether, under suitable assumptions, they converge to minimizers of the limit problem. Such a property is clearly crucial in view of numerical simulations. In particular, it paves the way to exploiting the calibration technique introduced in \cite{Marchese2016c,Marchese2016b} and extended to the discrete multi-material transport problem in \cite{MMT}. We begin with the following remark.

\begin{remark}[Multi-material fluxes with finite energy]
Even if the class of competitors for a given pair of compatible measures $(\mu^-, \mu^+)$ is always non-empty, the multi-material transport problem could be trivial: namely, it is possible that there is no multi-material flux between $\mu^-$ and $\mu^+$ with finite energy. In this case, we can say that every competitor is a solution. In Section \ref{s:finite_energy}, we give a sufficient condition on the multi-material transportation cost $\cost$ for the problem to be non-trivial, namely for every pair of compactly supported, compatible measures $(\mu^-,\mu^+)$ to admit a competitor with finite energy. Following \cite{Brancolini2017}, we call such multi-material transportation costs \emph{admissible} (see Definition \ref{def:admiss}).
\end{remark}

Without any assumptions on the cost functional $\cost$, stability results for branched transportation problems are not elementary (see e.g. \cite{Colombob,Colomboc,Colombo2017}). In \S \ref{ss:stability} we prove that, if the multi-material transportation cost is admissible, then the multi-material transport problem is stable.
\begin{theorem}[Stability of minimizers]\label{t:stability}
Let $\cost$ be an admissible multi-material transportation cost. Let $\mu^-_h,\mu^+_h$ be a sequence of pairs of compatible vector-valued measures in $\mathscr{M}(\R^n,\R^m)$ all supported on a common compact set $K$, and let $T_h$ be minimizers of the multi-material transport problem for the pair $(\mu^-_h,\mu^+_h)$. Assume, moreover, that 
$$\mu^\pm_h\wto\mu_{\infty}^\pm\quad\mbox{ and }\quad\sup_h\{\M(T_h)\}<\infty\,.$$
Then, up to subsequences, $T_h\wto T_\infty$, where $T_\infty$ is a minimizer of the multi-material transport problem for the pair $(\mu_\infty^-,\mu_\infty^+)$.
\end{theorem}

\section{Currents with coefficients in $\mathbb{R}^m$}\label{s:syllabus_euclidean}

As anticipated in the introduction, in order to tackle the multi-material transport problem we will take advantage of the formalism and the tools that are typical of the theory of currents. In this section we define currents with coefficients in $\R^m$ as the dual of a suitable space of differential forms. When we write \emph{classical} forms/currents, we refer to forms/currents with coefficients in $\R$ as in \cite[Section 4]{Federer_GMT}; a concise exposition, mostly sufficient to our aims, can also be found in \cite[2.5]{Brancolini2017}. The main goal of this section is to convey the idea that the properties of a current with coefficients in $\R^m$ can be studied by applying the results of the classical theory to the $m$-tuple of its \emph{components}.

\subsection{$\R^m$-valued covectors and forms}

A map 
\begin{equation}
\omega: \Lambda_k(\mathbb{R}^n) \times \R^m \to \mathbb{R}
\end{equation}
is an $\R^m$-valued $k$-covector on $\mathbb{R}^n$ ($1\leq k\leq n$) if:
\begin{enumerate}
\item $\forall \tau \in \Lambda_k(\mathbb{R}^n)$, $\omega(\tau,\cdot) \in (\R^m)^*$;
\item $\forall v \in \mathbb{R}^m$, $\omega(\cdot,v): \Lambda_k(\mathbb{R}^n) \to \mathbb{R}$ is a classical $k$-covector.
\end{enumerate}
The evaluation will be denoted with $\omega(\tau,v)$. The space of $\R^m$-valued $k$-covectors on $\R^n$ is denoted $\Lambda^k_{\R^m}(\mathbb{R}^n)$. We also set $\Lambda^0_{\R^m}(\mathbb{R}^n):=(\R^m)^*$.\\

For every $k$, the space of $\R^m$-valued $k$-covectors on $\mathbb{R}^n$ is a normed vector space when endowed with the norm $$\|\omega\| := \sup\{|\omega(\tau,\cdot)|\, \colon \, |\tau| \leq 1\,, \tau \text{ is simple}  \}\,,$$ where we have called \emph{simple} any $\tau \in \Lambda_k(\R^n)$ which can be given the form $\tau = \tau_1 \wedge \ldots \wedge \tau_k$ with each $\tau_l \in \R^n$. We can write the action of an $\R^m$-valued $k$-covector as $$\omega(\tau, v) = \sum_{j = 1}^mv^j\omega_j(\tau)\,,$$ 
where $v^j$ are the components of $v$ in the standard basis $\{e_1,\dots,e_m\}$ of $\R^m$, and, for $j=1,\dots,m$, the functions $\omega_j:\tau\mapsto \omega(\tau, e_j)$ are classical $k$-covectors, called the \emph{components} of $\omega$.\\

\noindent An \emph{$\R^m$-valued differential $k$-form on $\R^n$} is a map $$\omega : \R^n \to \Lambda^k_{\R^m}(\mathbb{R}^n).$$ We say that $\omega$ is smooth if every component $\omega_j$ is a classical smooth differential $k$-form. We denote by $$\mathscr{D}^{k}_{\R^m}(\R^n) := C_c^\infty(\R^n,\Lambda^k_{\R^m}(\mathbb{R}^n))$$ the vector space of smooth $\R^m$-valued differential $k$-forms on $\R^n$ with compact support.

\noindent The exterior differential of an $\R^m$-valued differential $k$-form $\omega$ on $\R^n$ is defined as the $\R^m$-valued differential $(k+1)$-form $d\omega$ on $\R^n$ whose components satisfy $(d\omega)_j=d(\omega_j)$, for every $j=1,\dots,m$. Moreover, the functional $\|\omega\|_{c} := \sup_{x \in \R^n} \|\omega(x)\|$ defines a norm on $\mathscr{D}^k_{\R^m}(\R^n)$, called the \emph{comass norm}.

\subsection {Currents with coefficients in $\R^m$}
Let $T$ be a linear functional on $\mathscr{D}^k_{\R^m}(\R^n)$. The components of $T$ are the linear functionals on $\mathscr{D}^k(\R^n) := C_c^\infty(\R^n,\Lambda^k(\mathbb{R}^n))$ defined by $T_j(\omega) := T(\hat\omega_j)$, where $\hat\omega_j$ is the $\R^m$-valued differential $k$-form on $\R^n$ whose $j$-th component coincides with $\omega$ and all other components are zero. We say that $T$ is continuous if and only if every component $T_j$ is a classical $k$-dimensional current.
The space of continuous linear functionals on $\mathscr{D}^k_{\R^m}(\R^n)$ is called the space of $k$-dimensional currents with coefficients in $\R^m$, and will be denoted $\mathscr{D}^{\R^m}_k(\R^n)$. We will sometimes write $T = \left( T_1,\dots,T_m \right)$ if $T \in \mathscr{D}_k^{\R^m}(\R^n)$ has components $T_1,\dots,T_m$.

\subsection{Boundary and mass}
Let $T\in \mathscr{D}^{\R^m}_k(\R^n)$, with $1\leq k\leq n$. The \emph{boundary} of $T$ is the current $\partial T \in \mathscr{D}^{\R^m}_{k - 1}(\R^n)$ defined by
\[
\partial T(\phi) := T(d\phi), \;\;\forall \phi \in \mathscr{D}^{k-1}_{\R^m}(\R^n).
\] 
Observe that $(\partial T)_j = \partial T_j$, namely that $\partial T = \left( \partial T_1,\dots,\partial T_m \right)$ if $T = \left( T_1,\dots,T_m\right)$. Also note that $\partial(\partial T) = 0$ for every $T \in \mathscr{D}_k^{\R^m}(\R^n)$.

The functional on $\mathscr{D}^k_{\R^m}(\R^n)$ defined by
\begin{equation} \label{mass for currents coeff Rm}
\M(T) := \sup\{|T(\omega)| : \|\omega\|_c \leq 1 \}
\end{equation}
is called \emph{mass}. A current $T$ with coefficients in $\R^m$ such that $\M(T) + \M(\partial T) < + \infty$ is called \emph{normal}. The space of $k$-dimensional normal currents on $\R^n$ with coefficients in $\R^m$ will be denoted by $\mathscr{N}^{\R^m}_k(\R^n)$.

\subsection{Currents with finite mass}\label{curr_finite_mass}
To every current $T\in \mathscr{D}^{\R^m}_k(\R^n)$ with finite mass one can associate a finite Borel measure $\|T\| \in \mathscr{M}_+ (\R^n)$ defined on open sets by
\[
\|T\|(\Omega) := \sup\{T(\omega)\, \colon \, \omega \in\mathscr{D}^k_{\R^m}(\R^n)\,, \|\omega\|_c\le 1\,, \spt(\omega)\subset \Omega\}.
\]

Let $T_j$ be the components of $T$, for $j=1,\dots,m$. Since every $T_j$ is a classical current with finite mass, Riesz's representation theorem allows to represent $T_j$ by integration, in the sense that
\begin{equation} \label{integrazione sulle componenti}
T_{j}(\omega) = \int_{\R^n} \langle\omega_{x},\tau_j(x)\rangle\, d\|T_j\|(x) \qquad \forall\, \omega \in \mathscr{D}^k(\R^n)\,,
\end{equation}
where the measure $\|T_j\|$ is defined on open sets as above by
\[
\|T_j\|(\Omega) := \sup\{T_j(\omega)\, \colon \, \omega \in\mathscr{D}^k(\R^n)\,, \|\omega\|_c\le 1\,, \spt(\omega)\subset \Omega\}\,,
\]
and where $\tau_j \colon \R^n \to \Lambda_k(\R^n)$ is a $\|T_j\|$-measurable classical $k$-vector field with mass $|\tau_j| = 1$ at $\|T_j\|$-almost every point. 

Now, the very definition of components easily implies that each measure $\|T_j\|$ is absolutely continuous with respect to $\|T\|$. Hence, by the Radon-Nikod\'ym theorem, there exist $\|T\|$-integrable, non-negative functions $t_j \colon \R^n \to \R$ with $t_j \leq 1$ $\|T\|$-almost everywhere such that $\|T_j\| = t_j \|T\|$. Let $\omega \in \mathscr{D}^k_{\R^m}(\R^n)$ be an $\R^m$-valued differential $k$-form, let $\omega_j$ be its components, and let again $\hat{\omega}_j$ denote the $\R^m$-valued differential $k$-form having $\omega_j$ as its $j$\textsuperscript{th} component and all the other components set equal to zero. Then, $\omega = \sum_{j} \hat{\omega}_j$, and we can compute
\begin{equation} \label{rappr per integrazione}
\begin{split}
T(\omega) &= T\left(\sum_{j=1}^m \hat{\omega}_j\right) = \sum_{j=1}^m T(\hat{\omega}_j) = \sum_{j=1}^m T_j(\omega_j) 
= \sum_{j=1}^m \int_{\R^n} \langle(\omega_j)_x,\tau_j(x)\rangle \, d\|T_j\|(x)\\
&= \int_{\R^n} \sum_{j=1}^m \omega_x(\tau_j(x), t_j(x) e_j) \, d\|T\|(x)\,.
\end{split}
\end{equation} 

If we let $A_T \colon \, \R^n \to \Lambda_k(\R^n) \otimes \R^m$ denote the $\|T\|$-measurable tensor field given by
\[
A_T := \sum_{j=1}^m \tau_j \otimes t_j e_j\,,
\]
we can formally write $T = A_T \|T\|$, which has to be understood in the sense that the action of $T$ on any $\omega \in \mathscr{D}^k_{\R^m}(\R^n)$ is as prescribed by equation \eqref{rappr per integrazione}. The $j$\textsuperscript{th} component of $T$ is represented by the $j$\textsuperscript{th} ``column'' of the tensor field $A_T$, as we can write $T_j = \tau_j\, t_j \|T\|$ for every $j=1,\dots,m$.

\subsection{Rectifiable currents with coefficients in $\R^m$} \label{ssec:correnti rettificabili}
Let $T$ be a $k$-current with coefficients in $\R^m$ and finite mass. 
We say that $T$ is \emph{rectifiable}, and we write $T \in \Rc_{k}^{\R^m}(\R^n)$, if it can be represented as $T = (\tau \otimes \theta)\, \Ha^k \trace E$, where:
\begin{itemize}
\item [(i)] $E$ is a countably $k$-rectifiable set in $\R^n$; 
\item [(ii)] $\tau$ is a \emph{simple} $k$-vector field with $|\tau(x)| = 1$ and $\tau(x)$ is orienting the approximate tangent space $\Tan(E,x)$ at $\mathcal{H}^k$-a.e. $x\in E$: that is, $\tau(x) = \tau_1(x) \wedge \ldots \wedge \tau_k(x)$ for an orthonormal basis $\{\tau_1(x), \ldots, \tau_k(x)\}$ of ${\rm Tan}(E,x)$;
\item [(iii)] $\theta: E \to \R^m$ belongs to $L^1(E,\R^m;\mathcal{H}^k \trace E)$.
\end{itemize}

The $k$-vector field $\tau$ and the function $\theta$ will be called an \emph{orienting $k$-vector} and the \emph{multiplicity vector} of $T$, respectively. A rectifiable current is called \emph{polyhedral} if $E$ is a finite union of $k$-dimensional simplexes with disjoint relative interiors and $\tau,\theta$ are constant on the relative interior of each simplex. The space of polyhedral $k$-currents with coefficients in $\R^m$ is denoted $\mathscr{P}_{k}^{\R^m}(\R^n)$. Notice that this definition is coherent with the content of Section \ref{s:mmtp}: in particular, the class $\mathscr{P}_{1}^{\R^m}(\R^n)$ contains precisely the polyhedral currents with coefficients in $\R^m$ introduced in Subsection \ref{multi_material_fluxes}.

\subsection{Flat norm and flat currents with coefficients in $\R^m$}\label{ss:flat}
The \emph{flat norm} is defined for $T \in \mathscr{P}^{\R^m}_k(\mathbb{R}^n)$ as:
\[
\mathbb{F}(T) :=\inf\{\M(S) + \M(R)\, \colon \, T = R + \partial S\,, R\in\mathscr{P}_k^{\mathbb{R}^m}(\mathbb{R}^n)\,, S\in\mathscr{P}_{k + 1}^{\mathbb{R}^m}(\mathbb{R}^n)\}.
\] 
The space of $k$-dimensional \emph{flat currents} with coefficients in $\mathbb{R}^m$, denoted $\mathscr{F}_k^{\mathbb{R}^m}(\R^n)$, is defined as the completion of the space $\mathscr{P}_k^{\mathbb{R}^m}(\R^n)$ in $\mathscr{D}^{\R^m}_k(\R^n)$ with respect to the flat norm. In particular, $T\in \mathscr{D}^{\R^m}_k(\R^n)$ is a $k$-dimensional flat current with coefficients in $\R^m$ if and only if every component $T_j$ of $T$ is a classical flat current. Moreover, it holds
\begin{equation} \label{comp flat estimates}
\mathbb{F}(T) \le \sum_{j = 1}^m\mathbb{F}(T_j),
\end{equation}
and clearly $\mathbb{F}(T_j) \le \mathbb{F}(T)$ for every $j$. Hence, the convergence in flat norm for currents with coefficients in $\R^m$ is equivalent to the convergence in flat norm of every component.

\begin{remark} \label{comparison}
We remark here that a current with coefficients in $\R^m$ is rectifiable, polyhedral, flat, normal or with finite mass if and only if all of its components are so. The classical theory of real currents, therefore, provides all the tools needed to work with currents with coefficients in $\R^m$. At the same time, flat currents with coefficients in $\R^m$ may be introduced as a particular instance of the theory of flat chains with coefficients in a normed Abelian group $\G$, as pioneered by Fleming \cite{Fleming} and extensively studied in the literature (see e.g. \cite{W_deformation,W_rectifiability,DPH_advanced,DPH_basic}), when $\G = (\R^m,+)$ equipped with the standard Euclidean norm. For the purposes of the present paper, the two approaches may be considered equivalent (see also Section \ref{currents} for more details on this point).
\end{remark}

\subsection{Currents and vector-valued measures}
The following fundamental result holds.
\begin{theorem}
If $K \subset \mathbb{R}^n$ is a compact set, and $r \geq 0$, then the set $$\{T \in \mathscr{N}^{\mathbb{R}^m}_k(\mathbb{R}^n)\,:\, \spt(T) \subset K\,,\, \M(T) + \M(\partial T) \leq r\}$$ is $\mathbb{F}$-compact in $\mathscr{F}^{\mathbb{R}^m}_k(\mathbb{R}^n)$.
\end{theorem}
For a proof of this theorem in the general case of currents with coefficients in groups see \cite[Lemma 7.4]{Fleming}. 

As we have seen in \S\ref{curr_finite_mass}, currents of finite mass with coefficients in $\R^m$ are identified with finite measures with values in $\Lambda_k(\R^n) \otimes \R^m$. Hence, the previous compactness result and the density with respect to the uniform topology of smooth forms in continuous ones guarantees that the flat convergence and the weak-$^*$ convergence of the associated measures coincide in a class of normal currents with equi-bounded masses and masses of the boundaries. Observe that in Definition \ref{def_cost_functional} no bound on the masses and masses of the boundaries is guaranteed on the sequence $T_h$ of polyhedral currents which are converging to $T$ in the flat norm. Hence, in principle, the functional $\en$ defined in there might not coincide with the lower semi-continuous relaxation made with respect to the {weak-$^*$} convergence of the vector-valued measures and of their distributional divergences. 

\subsection{Energy functional on spaces of currents with coefficients in $\R^m$}

Now that we have developed the terminology, let us rephrase the definition of the functional $\en$ in terms of currents with coefficients in $\R^m$. We will give the definition for currents of arbitrary dimension $k$, although the dimensions $k = 0$ and $k = 1$ are the only relevant in view of the application to the multi-material transport problem.

Let $\cost \colon \R^m \to \R$ be as in Definition \ref{def:mmtcost}. If $P \in \Po_{k}^{\R^m}(\R^n)$ is a polyhedral current of the form
 \[
 P = \sum_{\ell=1}^N (\tau_{\ell} \otimes \theta_{\ell}) \, \Ha^k \trace \sigma_\ell
 \]
for orienting $k$-vectors $\tau_\ell \in \Lambda_{k}(\R^n)$, multiplicities $\theta_\ell \in \R^m$ and non-overlapping $k$-dimensional simplexes $\sigma_\ell$, then we set
\begin{equation} \label{costo su poliedrali Rm}
\en(P) := \sum_{\ell=1}^N \cost(\theta_\ell) \, \Ha^k(\sigma_\ell)\,.
\end{equation}

The functional $\en$ is extended to $\F_k^{\R^m}(\R^n)$ by relaxation: if $T \in \F_{k}^{\R^m}(\R^n)$ then
\begin{equation} \label{costo su flat Rm}
\en(T) := \inf\left\lbrace \liminf_{h \to \infty} \en(P_h) \, \colon \, \{P_h\}_{h} \subset \Po_{k}^{\R^m}(\R^n) \mbox{ with } \flat(T - P_h) \to 0 \right\rbrace\,.
\end{equation}

\subsection{Decomposition of flat currents}

In the coming sections, we will use the following result, the proof of which can be found in \cite[Theorem 4.2]{Silhavy2008} in the case of classical flat currents; the proof in the case of a flat current $T$ with coefficients in $\R^m$ can be easily deduced by applying the classical result to each of the $m$ components of $T$.

\begin{thm}\label{dec}
Let $T\in \mathscr{F}_k^{\R^m}(\R^n)$ be a flat current with finite mass. Then, we can decompose
\[
T = T_{\rec} + T_{\diff}\,,
\]
where:
\begin{enumerate}
\item $T_{\rec} \in \Rc_{k}^{\R^m}(\R^n)$;
\item $\| T_{\diff}\|(A) = 0$ for every $k$-rectifiable set $A \subset \R^n$.
\end{enumerate}
We will call $T_{\rec}$ and $T_{\diff}$ the \emph{rectifiable part} and the \emph{diffuse part} of $T$ respectively. The decomposition is unique.
\end{thm}

\subsection{Representation of the cost of rectifiable currents}

The multi-material transportation energy $\E$ is defined via \eqref{costo su flat Rm} on all flat currents with coefficients in $\R^m$, and the finiteness of $\E(T)$ does not imply, in general, any further information on the geometry of $T$. There is, however, a simple condition which is both necessary and sufficient for all flat currents with finite mass and finite energy to be \emph{rectifiable}. For a multi-material cost $\cost$ as in Definition \ref{def:mmtcost}, we will define the directional derivatives
\[
\frac{\partial^+\cost}{\partial e_j}(0) := \lim_{t \to 0^+} \frac{\cost(t e_j)}{t}\,.
\]
The latter exist (possibly infinite) due to Lemma \ref{bigteo}(1) below.

\begin{prop} \label{p:rectifiability}
Let $\cost$ be as in Definition \ref{def:mmtcost}. The condition
\begin{equation} \label{e:infinite_derivative}
\frac{\partial^+\cost}{\partial e_j}(0) = + \infty \qquad \mbox{for every $j=1,\ldots,m$}
\end{equation}
holds if and only if every $T \in \mathscr{F}_1^{\R^m}(\R^n)$ with $\mass (T) + \E(T) < \infty$ is rectifiable.
\end{prop}

The previous result naturally generalizes the analogous result in the single-material framework (see Proposition 2.8 of \cite{Colombo2017a}). We postpone the proof of Proposition \ref{p:rectifiability} to Section \ref{s:prelim}.

\begin{remark}
Observe that the condition \eqref{e:infinite_derivative} is natural in branched transportation models. For instance, in the case $m=1$ it holds for the standard $\alpha$-mass functional originally studied by Xia \cite{Xia2003}, defined by $\cost (\theta) = |\theta|^\alpha$ for $\alpha \in \left( 0, 1 \right)$.
\end{remark}

In turn, the cost $\E(T)$ of a rectifiable current $T$ can be given an explicit integral representation. More precisely, we can prove the following theorem.

\begin{theorem}[Representation formula for rectifiable currents] \label{t:repr_networks}
Let $T \in \mathscr{F}_1^{\R^m}(\R^n)$ be a flat current with finite mass such that $T = T_{\rec} = (\tau \otimes \theta)\, \Ha^1 \trace E$. Then:
\begin{equation} \label{eq:rect_representation}
\E (T) = \int_{E} \cost(\theta (x)) \, d\Ha^1(x)\,.
\end{equation} 
\end{theorem}

Notice that \eqref{eq:rect_representation} reduces to \eqref{costo su poliedrali Rm} if $T$ is polyhedral. Theorem \ref{t:repr_networks} holds in fact in the much more general framework of arbitrary $k$-dimensional rectifiable chains with coefficients in a normed Abelian group ${\rm G}$ (see Section \ref{currents}), and we will prove it in that setting in Section \ref{s:representation}; see Theorem \ref{reprect}. The validity of Theorem \ref{reprect} was first stated by B. White in \cite{W_deformation}, but the proof was only sketched. A self-contained, complete proof of the result when $\G = \R$ was then proposed in \cite{Colombo2017a}, and has motivated further research on the topic of functionals on flat chains defined by relaxation (see e.g. the recent paper \cite{ChFeMe}). When we combine Proposition \ref{p:rectifiability} and Theorem \ref{t:repr_networks}, we have the following corollary.

\begin{cor} \label{cor:cost_char}
Let $\cost$ be as in Definition \ref{def:mmtcost}. The condition 
\begin{equation} \label{infinite slopes}
\frac{\partial^+\cost}{\partial e_j}(0) = + \infty \qquad \mbox{for every $j=1,\ldots,m$}
\end{equation}
holds if and only if 
\begin{equation}
\E (T) = 
\begin{cases}
\int_{E} \cost (\theta(x)) \, d\Ha^1(x) & \mbox{if $T = (\tau \otimes \theta)\, \Ha^1 \trace E \, \in \mathscr{R}_1^{\R^m}(\R^n)$}\,,\\
+ \infty &\mbox{if $T \in \mathscr{F}_1^{\R^m}(\R^n) \cap \{\mass (T) < \infty\} \setminus \mathscr{R}_1^{\R^m}(\R^n)$}\,.
\end{cases}
\end{equation}
\end{cor} 

\medskip

We remark that, in the single-material case, namely when $m=1$, Brancolini and Wirth have obtained, in \cite[Proposition 2.32]{Brancolini2017}, an explicit representation for the energy $\E (T)$ in all cases when \eqref{infinite slopes} fails, even when $T$ has a non-zero diffuse part $T_{\diff}$. It would be interesting to investigate the possibility to extend the results of \cite{Brancolini2017} to the multi-material case.

\section{Preliminaries to the existence theory} \label{s:prelim}

In this section we collect the main preliminary technical results towards the proofs of the main theorems of this work.

\subsection{Properties of the multi-material transportation cost}

We begin our program with the following lemma, which contains a detailed analysis of the properties of multi-material transportation costs.

\begin{lem}\label{bigteo}
Let $\mathcal{C} \colon \R^m \to \R$ be a multi-material transportation cost as in Definition \ref{def:mmtcost}. Then: 
\begin{enumerate}
\item the right-derivative of $\cost$ in the direction $v$ at $0$, defined as $\lim_{t \to 0^+}\frac{\cost(tv)}{t}$, and denoted with $\frac{\partial^+\cost}{\partial v}(0)$, exists (possibly infinite) for every $v \in \R^m$. Moreover, \begin{equation}\label{char}
 \frac{\partial^+\cost}{\partial v}(0) = \sup_{t>0}\frac{\mathcal{C}(tv)}{t} \qquad \forall \, v \in \R^m\,.
\end{equation}
\item the function $v \mapsto \frac{\partial^+\cost}{\partial v}(0)$ is positively $1$-homogeneous, even, strictly positive on $\R^m \setminus \{0\}$, subadditive, and monotone with respect to the partial order $\preceq$ on $\R^m$;
\item the set $V:=\{v\in \R^m\,:\, \frac{\partial^+\cost}{\partial v}(0) < +\infty\}$ is a vector subspace of $\mathbb{R}^m$, and a basis for $V$ is given by $\mathcal{B}:=\{e_{j}\,:\, \frac{\partial^+\cost}{\partial e_j}(0)< +\infty\}$;
\item the function $v\in V \mapsto \frac{\partial^+\cost}{\partial v}(0)$ is bounded on $\mathbb{S}^{m - 1}\cap V$, with estimates:
\begin{equation}\label{bound}
\frac{\partial^+\cost}{\partial v}(0)\le \sum_{j \colon e_j \in \mathcal{B}} |v_j|\frac{\partial^+\cost}{\partial e_j}(0) \leq m \frac{\partial^+\cost}{\partial v}(0)\, \qquad \forall \, v = \sum_{j \colon e_j \in \mathcal{B}} v_j e_j \in \Sf^{m-1} \cap V\,;
\end{equation}
\item $v\mapsto\frac{\partial^+\cost}{\partial v}(0)$ is continuous on $V$;
\item for every $\delta >0$ there exists a constant $c = c(\delta)>0$ such that 
\begin{equation} \label{norm-cost-ratio-bound}
|v| \leq c \, \cost(v) \qquad\forall\, v\in \overline{B(0,\delta)}\,.
\end{equation}
\end{enumerate}
\end{lem}

\begin{proof}
The proof of $(1)$ can be found in \cite[Theorem 16.3.3]{Kuczma2009}. In order to apply that theorem, we only need to show that, for every $v \in \R^m$, it is either $\lim_{t \to 0^+}\mathcal{C}(tv) = 0$ or $\liminf_{t\to 0^+}\mathcal{C}(tv)>0$. To prove that for a multi-material cost this is always verified, we simply observe that the function $t \mapsto \mathcal{C}(tv)$ is non-decreasing for $t > 0$. Therefore, the limit $\lim_{t\to0^+} \cost(tv)$ exists, and, since $\mathcal{C}$ is a non-negative function, it is either $0$ or positive. 

We now prove $(2)$. The fact that the directional derivatives are positively $1$-homogeneous functions is due to basic properties of the limit. To simplify the notation, we denote $\frac{\partial^+\cost}{\partial v}(0)$ with $f(v)$. $f$ is even because such is $\mathcal{C}$. Its strict positivity is a direct consequence of the strict positivity of $\mathcal{C}$ and \eqref{char}. Subadditivity is checked in the following way. Write
\[
\cost(t(v_1 + v_2)) \le \cost(tv_1) + \cost(tv_2) \qquad \forall t>0\,\,,  \forall\, v_1,v_2 \in \R^m.
\]
Dividing by $t$ and taking the limit we get the desired inequality. Finally, by the monotonicity of $\cost$, we also have
\[
\cost(tv_1)\leq \cost(tv_2) \qquad \forall t>0\,,v_1 \preceq v_2.
\]
Once again passing to the incremental quotients we infer the same inequality for $f$. 

To prove $(3)$ it is sufficient to write every vector $v\in\R^m$ in components as $v = \sum_{j = 1}^m v_je_j$ and observe that
\[
f(v) \le \sum_{j = 1}^mf(v_je_j),
\]
by subadditivity. Observe that, since $f$ is even, $f(\sign(v_j)e_j) = f(e_j),\forall j\in\{1,\dots,m\}$. Therefore, by subadditivity, positive homogeneity and evenness we deduce
\begin{equation}\label{boundp}
f(v) \le \sum_{j = 1}^m|v_j|f(\sign(v_j)e_j) = \sum_{j = 1}^m|v_j|f(e_j) .
\end{equation}
This inequality proves that $\spn(\mathcal{B})\subseteq V$. On the other hand, suppose that $v = \sum_{j = 1}^m v_je_j$ is such that $v_\ell \neq 0$ for some $\ell$ such that $e_\ell\notin\mathcal{B}$. We need to show that $$f(v) = +\infty.$$ This is a consequence of the monotonicity of $f$. Indeed, since $f$ is even, we can always suppose that $v_\ell$ is positive, and write
\begin{equation}\label{boundp2}
f(v_\ell e_\ell)\le f(v).
\end{equation}
By homogeneity we infer that $f(v_\ell e_\ell) = v_\ell f(e_\ell) = +\infty$. We obtain that $V = \spn(\mathcal{B})$ and we conclude the proof of $(3)$. 

Conclusion $(4)$ has already been proved with equations $\eqref{boundp}$ and $\eqref{boundp2}$.  We will now prove $(5)$. We will let $r := \dim V$ and identify $V$ with $\mathbb{R}^r$. Consider a sequence $\{v^h\}_{h \in \Na}$ in $V$ converging to $v\in V$. We prove the continuity of $f$ by proving separately upper and lower semi-continuity. The upper semi-continuity is proved as follows. By subadditivity,
\[
\cost(tv^h)\le \cost(t(v^h - v)) + \cost(tv) \qquad \forall h\in\mathbb{N}\,.
\]
Denote with $w^h := v^h - v$. Then, using equation \eqref{char} and the homogeneity of $f$, we have:
\[
\cost(t(v^h - v)) \leq f\left(\frac{w^h}{|w^h|}\right)t |w^h| \qquad \forall\, t>0\,, \forall\, h\in\mathbb{N}\,.
\]
Since, by $(4)$, $f$ is bounded by a constant $L$ on the sphere of $V$, we obtain:
\[
\sup_{t>0}\frac{\cost(t(v^h - v))}{t} \leq f\left(\frac{w^h}{|w^h|}\right)|w^h|\leq L|w^h| \qquad \forall\, h\in\mathbb{N}\,.
\]
Hence
\[
\frac{\cost(tv^h)}{t}\le L|w^h| + \frac{\cost(tv)}{t} \le L|w^h| + \sup_{t>0}\frac{\cost(tv)}{t} = L|w^h| + f(v) \qquad \forall\, h\in\mathbb{N}\,.
\]
Taking the limit as $t\to 0^+$, we get
\[
f(v^h) \le L|w^h| + f(v) \qquad  \forall\, h\in\mathbb{N}\,,
\]
and finally
\[
\limsup_h f(v^h)\le f(v)\,.
\]
Lower semi-continuity is completely analogous. It suffices to use subadditivity to write, $\forall h\in\mathbb{N}$,
\[
\cost(tv)\le \cost(t(v^h - v)) + \cost(tv^h),
\]
and repeat the proof of the upper semi-continuity. 

We will now prove $(6)$. Since $\frac{\partial^+\cost}{\partial e_j}(0)$ is strictly positive for every $j=1,\dots,m$, there exist $\alpha >0$ and $\e >0$ such that, whenever $v = \sum_{j=1}^m v_j e_j \in \R^m$ satisfies $\abs{v} < \e$, one has:
\[
\abs{v_j} \leq \alpha\, \cost(\abs{v_j} e_j) = \alpha\,  \cost(v_j e_j) \qquad \forall\, j=1,\dots,m\,.
\]
By the monotonicity of the cost, we can write
\[
\abs{v_j} \leq \alpha \, \cost(v) \qquad \forall\, j\in\{1,\dots, m\}\,,
\]
and, by taking the maximum over $j$,
\[
\abs{v}_{\infty} \leq \alpha \, \cost(v) \qquad \forall \, v \in B(0,\e)\,,
\]
which evidently implies
\begin{equation} \label{controllo dall'alto palla piccola}
\abs{v} \leq \sqrt{m}\, \alpha\, \cost(v) \qquad \forall \, v \in B(0,\e)\,.
\end{equation}

On the other hand, since $\cost$ is lower semi-continuous, the function $v\mapsto \frac{|v|}{\mathcal{C}(v)}$ is upper semi-continuous on every set not containing the origin. Therefore, the inequality in \eqref{controllo dall'alto palla piccola} holds, with a constant depending on $\delta$ and $\e$, also in $\overline{B(0,\delta)} \setminus B(0,\e)$. This completes the proof.
\end{proof}

\subsection{Properties of the energy functional on flat currents} 

We turn now our attention to study the properties of the energy functional, regarded as a map
\[
\en \colon \F_{1}^{\R^m}(\R^n) \to \left[0,+\infty\right]\,,
\]
as defined in \eqref{costo su poliedrali Rm}-\eqref{costo su flat Rm} when $k=1$.

Here and in the sequel, we are going to adopt the following notation. 

\begin{notazioni}[Lift of a component] \label{defnot:lift}
If $T \in \mathscr{D}_k^{\R^m}(\R^n)$ is a current with coefficients in $\R^m$, and $T_j$ is the $j$\textsuperscript{th} component of $T$, we let $\hat{T}_j$ denote the current in $\mathscr{D}_k^{\R^m}(\R^n)$ having the $j$\textsuperscript{th} component equal to $T_j$ and all other components set equal to zero. We will call $\hat{T}_j$ the \emph{lift} of the component $T_j$ to $\mathscr{D}_k^{\R^m}(\R^n)$. Note that any current with coefficients in $\R^m$ can be decomposed into the sum of the lifts of its components, namely
\begin{equation} \label{decomposition in lifts}
T = \sum_{j=1}^m \hat{T}_j\,, \qquad T \in \mathscr{D}_{k}^{\R^m}(\R^n)\,.
\end{equation}
\end{notazioni}

The following lemma states that, if a flat current $T$ coincides with the lift of one of its components, then one can find a (polyhedral) recovery sequence for its energy enjoying the same property.
\begin{lem}\label{redu}
Let $T \in \F_1^{\R^m}(\R^n)$ be a flat current such that $T = \hat{T}_j$ for some $j \in \{1,\dots,m\}$. Then, there exists a sequence $\{P^h\}_{h \in \mathbb{N}}$ with $P^h \in \mathscr{P}_{1}^{\R^m}(\R^n)$ such that
\begin{equation} \label{coefficienti sottospazio}
P^h = \widehat{P^h}_j \qquad \mbox{for every } h \in \mathbb{N}\,,
\end{equation}
and
\begin{equation} \label{convergenza}
\flat(P^h - T) \to 0\,, \quad \en(P^h) \to \en(T) \qquad \mbox{ as } h \to \infty\,.
\end{equation}
\end{lem}

\begin{proof}
Let $T^h$ be a recovery sequence for $\en(T)$: namely, assume that $\{T^h\}_h$ is a sequence of polyhedral currents with coefficients in $\R^m$ such that
\[
\flat(T - T^h) \to 0 \quad \mbox{ and } \quad \en(T^h) \to \en(T) \quad \mbox{as } h \to \infty\,.
\]
Recall that the flat norm of a current with coefficients in $\R^m$ is comparable to the sum of the flat norms of its components, cf. \S \ref{ss:flat}. For every $h \in \Na$, simply set $P^h := \widehat{T^h}_j$. In other words, $P^h$ is obtained by projecting the multiplicities of $T^h$ onto the subspace ${\rm span}\left[e_j\right] \subset \R^m$, so that if $T^h = \left( T^h_1,\dots,T^h_j,\dots,T^h_m\right)$ then $P^h = \left( 0, \dots, T^h_j,\dots,0 \right)$. Obviously, \eqref{coefficienti sottospazio} holds by construction, and $\flat(P^h - T) \to 0$ as $h \to \infty$. Finally, the procedure does not increase the energy, by monotonicity of the multi-material transportation cost $\cost$. Hence:
\[
\en(T) \leq \liminf_{h\to\infty} \en(P^h) \leq \limsup_{h\to\infty} \en(P^h) \leq \limsup_{h \to \infty} \en(T^h) = \en(T)\,,
\]
which completes the proof.
\end{proof}

\begin{prop}\label{lem_mon}
Let $\cost$ be a multi-material transport cost. The associated energy $\en$ has the following properties:
\begin{enumerate}
\item If $T \in \F_1^{\R^m}(\R^n)$, then
\[
\en(\hat{T}_j) \leq \en(T) \qquad \mbox{for every } j=1,\dots,m\,.
\]
\item If $T,S\in \mathscr{F}_1^{\R^m}(\R^n)$, then
\[
\en(T + S) \le \en(T) + \en(S).
\]
In particular,
\begin{equation} \label{energia componenti comparabile}
\en(T) \leq \sum_{j=1}^m \en(\hat{T}_j) \leq m\, \en(T)\,.
\end{equation}

\item If $A$ and $B$ are disjoint Borel sets and $T\in \mathscr{F}_1^{\R^m}(\R^n)$ has finite mass, then
\begin{equation} \label{spezzamento}
\en(T \trace (A \cup B)) = \en(T\trace A) + \en(T\trace B).
\end{equation}
In particular, 
\begin{equation} \label{spezzamento_bello}
\en(T) = \en(T_{\rec}) + \en(T_{\diff})\,.
\end{equation}
\end{enumerate}
\end{prop}

\begin{proof}

The proof of $(1)$ is analogous to that of Lemma \ref{redu}, by means of projecting the multiplicities of a polyhedral recovery sequence onto the direction ${\rm span}[e_j]$. Concerning $(2)$, observe that the subadditivity of $\en$ holds on polyhedral currents simply by the subadditivity of the cost $\cost$. The result naturally extends to flat currents by approximation, and thus \eqref{energia componenti comparabile} is an immediate consequence of \eqref{decomposition in lifts} and $(1)$. Finally, the proof of \eqref{spezzamento} can be found in \cite[6.1(3)]{W_deformation}. If we apply Theorem \ref{dec} to decompose
\[
T = T_{\rec} + T_{\diff}\,,
\]
then evidently the measures $\|T_{\rec}\|$ and $\|T_{\diff}\|$ are mutually singular, and in fact $T_{\rec} = T \trace E$ and $T_{\diff} = T \trace E^c$ for some $1$-rectifiable set $E \subset \R^n$. Hence, \eqref{spezzamento_bello} readily follows from \eqref{spezzamento}.
\end{proof}

\subsection{Rectifiability of currents with finite energy: proof of Proposition \ref{p:rectifiability}}

The technical lemmas of the previous paragraphs are sufficient to show the validity of Proposition \ref{p:rectifiability}. Notice that, using the terminology of Lemma \ref{bigteo}, the condition in \eqref{e:infinite_derivative} can be rephrased as $\mathcal{B} = \emptyset$, or equivalently $V = \{0\}$. The proof will be obtained by working in components as a corollary of the corresponding statement in the case $m=1$, which can be found in \cite[Proposition 2.8]{Colombo2017a}.

\begin{proof}[Proof of Proposition \ref{p:rectifiability}]
We first claim that the condition in \eqref{e:infinite_derivative} is necessary for every flat current with finite mass and energy to be rectifiable. Indeed, suppose that \eqref{e:infinite_derivative} does not hold, and assume without loss of generality that $\frac{\partial^+\cost}{\partial e_1}(0) < \infty$. Define a sequence $\{P^h\}_h$ of polyhedral currents of the form $P^h = \left( P^h_1, 0, \ldots,0 \right)$, where $P^h_1$ are as in the corresponding counterexample in \cite[Proof of Proposition 2.8]{Colombo2017a}. As $h \uparrow \infty$, $P^h$ converges to $T = \left( T_1,0,\ldots,0 \right) \in \mathscr{F}_1^{\R^m}(\R^n) \cap \{\mass(T) <\infty\} \setminus \mathscr{R}_1^{\R^m}(\R^n)$ and $\E(T) \leq \liminf_{h\to\infty} \E_1(P^h_1) < \infty$, where $\E_1$ is the single-material transportation energy induced by $\cost_1(\theta) := \cost (\theta, 0, \ldots, 0)$.\\

Conversely, we show that if \eqref{e:infinite_derivative} holds then every $T$ with $\mass (T) + \E (T) < \infty $ is rectifiable. To this aim, it suffices to prove that every component $T_j$ is rectifiable. Fix $j$, and let $\hat T_j$ be the lift of the component $T_j$. By Proposition \ref{lem_mon}, $\mass (\hat T_j) + \E(\hat T_j) < \infty$. By Lemma \ref{redu}, there exists a recovery sequence for $\E(\hat T_j)$ of the form $P^h = \left( 0, \ldots, P^h_j,\ldots, 0\right)$, so that $\E(P^h) = \E_j(P^h_j)$, $\E_j$ being the single-material cost functional induced by $\cost_j(t) := \cost (t e_j)$ for $t \in \R$. Since $\flat(P^h_j - T_j) \to 0$, we have that $\E_j(T_j) \leq \liminf_{h\to\infty}  \E(P^h) = \E(\hat T_j) < \infty$. Since $\cost_j$ satisfies $\lim_{t\to 0^+} \frac{\cost_j (t)}{t} = +\infty$, we conclude from \cite[Proposition 2.8]{Colombo2017a} that $T_j$ is rectifiable.
\end{proof}

\subsection{Monotonicity of the energy}

In Proposition \ref{lem_mon} $(1)$ we have concluded that any \emph{component} of $T$ (or, better said, the lift of any component of $T$) has less energy than $T$. Our next goal is to obtain an analogous result when we look at \emph{pieces} of $T$ rather than components. Let us first define what a piece of a current is.

\smallskip
Let $T\in\mathscr{D}_1(\R^n)$ be a classical current with finite mass. We say that $T'\in\mathscr{D}_1(\R^n)$ is a \emph{piece} of $T$ if 
$$\mass(T)=\mass(T')+\mass(T-T').$$
By \cite[Section 4.1.7]{Federer_GMT} one can see that if $T$ is represented by integration as $T=\tau\mu$, with $\tau$ a Borel vector field in $\R^n$ and $\mu \in \mathscr{M}_{+}(\R^n)$, then a piece $T'$ of $T$ can be written as $T'=\lambda\tau\mu$, for a Borel function $\lambda:\R^n\to[0,1]$.
For $T,T'\in\mathscr{F}_1^{\R^m}(\R^n)$ we say that $T'$ is a piece of $T$ if every component $T'_j$ of $T$ is a piece of the corresponding component $T_j$ of $T$, for $j=1,\dots,m$.

The following is the anticipated monotonicity result for pieces of a flat current with coefficients in $\R^m$, which will play a fundamental role in our existence theory.
\begin{prop} \label{prop:monotonicity}
Let $T\in \mathscr{F}_1^{\R^m}(\R^n)$ have finite mass, and let $T'$ be a piece of $T$. Then $\mathbb{E}(T')\leq\en(T)$.
\end{prop}

The result in Proposition \ref{lem_mon} $(1)$ stems from the analogous result valid in the case of polyhedral currents, which in turn is an immediate consequence of the monotonicity properties of the function $\cost$. The proof of Proposition \ref{prop:monotonicity} will also follow from a polyhedral approximation argument; the difficulty here lies precisely in the construction of a suitable polyhedral approximation. The \emph{slicing theory} for classical normal currents is one of the ingredients that we are going to need for the proof. We recall here that if $S \in \mathscr{N}_k(\R^n)$, $f \colon \R^n \to \R$ is a Lipschitz function, and $r$ is a real number, then $\langle S, f, r \rangle$ denotes the \emph{slice} of the current $S$ via the map $f$ at level $r$. Intuitively, $\langle S, f, r \rangle$ may be thought of as the $(k-1)$-dimensional current obtained by ``intersecting'' $S$ with the level set $\{f(x)=r\}$: this interpretation is actually entirely correct (modulo specifying the orientation of the resulting object) when $S$ is the current associated with a smooth $k$-surface in $\R^n$.  

\begin{proof}[Proof of Proposition \ref{prop:monotonicity}]

Following the discussion in \S \ref{curr_finite_mass}, write $T=(\tau_1\mu,\dots,\tau_m\mu)$, with $\mu \in \mathscr{M}_{+}(\R^n)$ and $\tau_1,\dots,\tau_m$ Borel vector fields in $\R^n$ satisfying $|\tau_j| \leq 1$ $\mu$-almost everywhere for every $j$. Let $\lambda_1,\dots,\lambda_m:\R^n\to[0,1]$ be Borel functions such that $T'_j=\lambda_j\tau_j\mu$, for $j=1,\dots,m$. Fix $\varepsilon>0$ and $k\in\N$. For every $\ell=(\ell_1,\dots,\ell_m)\in\{0,\dots,k\}^m$, denote 
$$D_\ell:=\left\{x\in\R^n\,:\,\lambda_j(x)\in\left[\frac{\ell_j}{k},\frac{\ell_j+1}{k}\right),\, j=1,\dots,m\right\}\,.$$
For every $\ell$, we also let $K_\ell\subset D_\ell$ be a compact set such that 
\begin{equation}\label{cicciacompact}
\mu(D_\ell\setminus K_\ell)<\frac{\varepsilon}{(k+1)^m}.
\end{equation} 
Since $\mu(\R^n\setminus\bigcup_\ell D_\ell)=0$, it follows from \eqref{cicciacompact} that 
\begin{equation}\label{ciccia2}
\mu\left(\R^n\setminus \bigcup_\ell K_\ell\right)<\varepsilon.
\end{equation}

Observe that, since the $D_\ell$'s are finitely many disjoint sets, the mutual distances between the $K_\ell$'s are bounded from below by a number $2\rho_0\leq 1$. Therefore, for every $\rho<\rho_0$ the open sets 
$$A^\rho_\ell:=\{x\in\R^n\,:\,{\rm dist}(x,K_\ell)<\rho\}$$ 
are mutually disjoint.
For every $j=1,\dots,m$, for every $\rho<\rho_0$, we denote 
$$T''_j:=\sum_{\ell}\frac{\ell_j}{k}\,T_j\trace A^\rho_\ell\,,$$
and we let $T''\in \mathscr{F}_1^{\R^m}(\R^n)$ be the current with components $T''_1,\dots,T''_m$.

It follows from \eqref{ciccia2} and the definition of $D_\ell$ that, for $j=1,\dots, m$, for every $\rho<\rho_0$ we have
\begin{equation}\label{primastima}
\begin{split}
\mass(T''_j-T'_j)&\leq 2\,\mu\left(\R^n\setminus\bigcup_\ell K_\ell\right)+\mass\left((T''_j-T'_j)\trace\bigcup_\ell K_\ell\right)\leq 2\varepsilon + \sum_\ell\mass(T''_j\trace K_\ell-T'_j\trace K_\ell)\\ &\leq 2\varepsilon+\frac{1}{k}\mu(\R^n)\,.
\end{split}
\end{equation}

Now, let us consider a recovery sequence $\{P^h\}_{h\in\N}$ for the energy of $T$, namely $P^h\in\mathscr{P}_1^{\R^m}(\R^n)$ with $\flat(P^h-T)\to 0$ and $\en (P^h)\to \en (T)$. We denote, for $j=1,\dots,m$, for every $h\in\N$, for $\rho<\rho_0$,
$$Q^h_j:=\sum_{\ell}\frac{\ell_j}{k}\,P^h_j\trace A^\rho_\ell\,,$$
and we let $Q^h\in \mathscr{R}_1^{\R^m}(\R^n)$ be the current with components $Q^h_1,\dots,Q^h_m$. Observe that every $Q^h_j$ is supported on a relatively open subset $U^h_j(\rho)$ of the support of $P^h_j$. Therefore, for every $h$ and for $j=1,\dots,m$, $U_j^h(\rho)$ is the union of at most countably many line segments. Hence, there exists a set $E_j^h(\rho)\subset U_j^h(\rho)$, which consists of finitely many line segments, such that $P'^h_j:=Q^h_j\trace E^h_j(\rho)\in \mathscr{P}_1(\R^n)$ and 
\begin{equation}\label{quasipoly}
\mass(Q^h_j\trace (U^h_j(\rho)\setminus E^h_j(\rho)))<\e\,, \quad\mbox{for every $h$, for $j=1,\dots,m$ and for every $\rho<\rho_0$}.
\end{equation}
Finally, for every $h\in\N$, we denote $P'^h\in\mathscr{P}_1^{\R^m}(\R^n)$ the current with components $P'^h_1,\dots,P'^h_m$.
Clearly, $P'^h$ is a piece of $P^h$ for every $h$, and thus, by monotonicity of $\cost$, we have $\en(P'^h)\leq\en(P^h)$ for every $\rho<\rho_0$.

We claim that there exists $\rho<\rho_0$ such that
\begin{equation}\label{e:claim}
\flat(P'^h-T')\leq C(k,\varepsilon)\quad\mbox{for infinitely many $h\in\N$}\,, \tag{CLAIM}
\end{equation}
where $C(k,\varepsilon)$ vanishes in the limit as $k\to\infty$ and $\varepsilon \to 0$. This would imply that $\en(T')\leq \en(T)$ and conclude the proof.

As we observed in \eqref{comp flat estimates} we have $\flat(P'^h-T')\leq\sum_{j=1}^m\flat(P'^h_j-T'_j)$, hence it is sufficient to prove that there exist $\rho<\rho_0$ and an infinite set of indices $H:=\{h_1,h_2,\ldots\}$, such that for each $j=1,\dots,m$, we have

$$\flat(P'^h_j-T'_j)\leq C(k,\varepsilon)\quad\mbox{for every $h\in H$}.$$
For every $h \in \mathbb{N}$, for every $j\in\{1,\dots,m\}$, we have by \eqref{primastima} and \eqref{quasipoly}
\[
\begin{split}
\flat(P'^h_j-T'_j)&\leq\flat(Q^h_j-T''_j)+\mass(Q^h_j-P'^h_j)+\mass(T'_j-T''_j)\leq \flat(Q^h_j-T''_j)+3\varepsilon+\frac{1}{k}\mu(\R^n) \\
&\leq\sum_\ell\frac{\ell_j}{k}\flat((T_j-P^h_j)\trace A^\rho_\ell)+3\varepsilon+\frac{1}{k}\mu(\R^n)\,.
\end{split}
\]
Hence, in order to prove the claim it suffices to show that there exist a positive constant $C$ (independent of $\varepsilon$ and $k$) and a radius $\rho<\rho_0$, such that for every $\ell$ \[\flat((T_j-P^h_j)\trace A^\rho_\ell)<C\frac{\varepsilon}{(k+1)^m} \quad \mbox{for every $j=1,\dots,m$ and for infinitely many $h\in\N$}\,.\]

In order to see this, fix an index $j\in\{1,\dots,m\}$. For every $p =1,2,\dots$ let $h_p \in \mathbb{N}$ (with $h_p > h_{p-1}$ for $p\geq 2$) be such that 
\[
\flat(T_j - P_j^h) < \frac{\e}{(k+1)^m} \frac{\rho_0}{2^{p+2}} \quad \mbox{for every $h \geq h_p$}\,.
\]
Then, by the well-known characterization of the flat norm of classical currents (cf. \cite[Section 4.1.12]{Federer_GMT}), for every $h \geq h_p$ there are currents $R^h \in \mathscr{D}_1(\R^n)$ and $S^h \in \mathscr{D}_2(\R^n)$ such that 
\begin{equation} \label{flat competitors}
T_j - P_j^h = R^h + \partial S^h\,, \qquad \mass(R^h) + \mass(S^h) \leq \frac{\e}{(k+1)^m} \frac{\rho_0}{2^{p+1}}\,.
\end{equation}
Observe that \eqref{flat competitors} implies that $S^h$ is a normal current. Hence, by the classical slicing formula for normal currents (see \cite[Section 4.2.1]{Federer_GMT}) we have for every $\ell\in\{0,\dots,k\}^m$ that
\begin{equation} \label{flat finale}
\begin{split}
(T_j - P_j^h) \trace A^{\rho}_\ell &= R^h \trace A^\rho_\ell + (\partial S^h) \trace A^\rho_\ell \\
&= R^h \trace A^\rho_\ell - \langle S^h, {\rm dist}(\cdot, K_\ell), \rho \rangle + \partial (S^h \trace A^\rho_\ell)
\end{split}
\end{equation}
for a.e. $\rho < \rho_0$. On the other hand, since (see again \cite[Section 4.2.1]{Federer_GMT})
\[
\int_{0}^{\rho_0} \mass(\langle S^h, {\rm dist}(\cdot, K_\ell),\rho\rangle)\, d\rho \leq \mass(S^h \trace A^{\rho_0}_\ell) \leq \frac{\e}{(k+1)^m} \frac{\rho_0}{2^{p+1}}\,,
\]
then there exists a set $I^h \subset \left(0,\rho_0\right)$ of length $|I^h| \leq \frac{\rho_0}{2^{p+1}m}$ such that
\begin{equation} \label{massa slice}
\mass(\langle S^h, {\rm dist}(\cdot, K_\ell), \rho\rangle) \leq \frac{m\e}{(k+1)^m} \quad \mbox{for every $\rho \in \left(0,\rho_0\right)\setminus I^h$, for every $h \geq h_p$}\,.
\end{equation}

Set $H := \{h_p\}_{p\geq 1}$, $I := \bigcup_p I^{h_p}$, and observe that $|I| \leq \frac{\rho_0}{2m}$. Hence, if we choose $\rho \in \left(0,\rho_0\right) \setminus I$ we conclude that \eqref{massa slice} holds true for every $h \in H$. In turn, this allows us to estimate from \eqref{flat finale} that 
\[
\flat((T_j - P_j^h)\trace A^\rho_\ell) \leq \mass(R^h) + \mass(\langle S^h,{\rm d}_{K_\ell},\rho \rangle) + \mass(S^h) \leq \frac{2m\e}{(k+1)^m}
\]
for every $h \in H$, thus completing the proof.
\end{proof}

\section{Proof of the existence Theorem \ref{t:existence}}\label{s:existence}

The proof is by direct methods. Since we know that (by definition) the energy $\E$ is lower semi-continuous with respect to the convergence in flat norm, our goal is to embed the minimization problem introduced in Definition \ref{defn:mmtprob} in a class of 1-currents with coefficients in $\R^m$ which is compact with respect to the topology induced by the flat norm. Let $\{T^h\}_{h\in\N}$ be a sequence of multi-material fluxes between $\mu^-$ and $\mu^+$ which is minimizing the energy $\E$. The sequence $T^{h}$ consists of normal 1-currents with coefficients in $\R^m$ having a common boundary. Moreover, we can assume that the currents $T^{h}$ are all supported on a common compact set (because the push-forward with respect to the closest-point projection from $\R^n$ onto a convex polytope containing the support of $\mu^--\mu^+$ does not increase the energy $\E$, by the subadditivity of the cost). We can also assume that the infimum of the energies $\E(T^h)$ is finite. Nevertheless, the (finite) masses of the $T^{h}$ might in principle be unbounded along the sequence. We will prove that one can perform an operation on each $T^{h}$ (which roughly speaking consists in \emph{removing all its cycles}) which preserves the boundary and does not increase the energy. Moreover the modified currents $T'^{h}$ satisfy
$$\M(T'^{h})\leq C\E(T'^{h})\leq C \E(T^{h})\,.$$
This bound recasts the problem in a compact regime, hence the minimality of each element of the (non-empty) class of subsequential limits of $\{T'^{h}\}_{{h}\in\N}$ is guaranteed by direct methods.

\subsection{Removing cycles}
Let $T\in\mathscr{N}_1(\R^n)$ be a (classical) 1-dimensional normal current. We say that $S\in\mathscr{N}_1(\R^n)$ is a \emph{cycle contained in} $T$ if 
\begin{equation}\label{e:contained_cicle}
\partial S = 0 \qquad \mbox{and} \qquad \Mass(T)=\Mass(T-S)+\Mass(S).
\end{equation}
In other words, a cycle contained in $T$ is any piece of $T$ with zero boundary. We say that $T$ is \emph{acyclic} if there is no non-zero cycle contained in $T$.  
By \cite[Proposition 3.8]{Paolini2012} one can identify the {largest} cycle contained in a normal current $T$, i.e. a cycle $S$ contained in $T$ such that $T':=T-S$ is acyclic. We will call $T'$ the \emph{acyclic part} of $T$. First of all, let us observe that, since $T'$ is a piece of $T$, if $T=\vec T\|T\|$, with unit orientation $\vec T$, then $T'$ can be written as 
\begin{equation}\label{e:rep_acyclic}
T'=\lambda\vec T\|T\|,
\end{equation}
where $\lambda:\R^n\to[0,1]$ is a measurable function. Also note that evidently $\partial T' = \partial T$, since $T - T'$ has zero boundary.

The following lemma contains a crucial observation for the proof of the existence theorem.
\begin{lem}\label{l:bound_rect}
Let $T$ be an acyclic normal 1-current, and let $R=\tau\,\theta\,\Haus^1\trace E\in\mathscr{R}_1(\R^n)$ be its rectifiable part, according to the decomposition Theorem \ref{dec}. Then $|\theta(x)|\leq\frac{1}{2}\Mass(\partial T)$ for $\Haus^1$-a.e. $x\in E$.
\end{lem}
\begin{proof}
Without loss of generality, write $T=\tau (|\theta|\Haus^1\trace E+\mu)$, where $\tau$ is unitary and $\mu(E')=0$ for every 1-rectifiable set $E'$. The proof is a small variation of the proof of Prop. 3.6 (2) of \cite{Colombo2017}, and we refer to that paper for the relevant notation. We just recall that one can identify $T$ with a positive measure $\pi$  on the space $\Lip$ of Lipschitz parametrized curves, in the sense that
\[
T (\omega) = \int_{\Lip} \llbracket \gamma \rrbracket (\omega) \,d\pi(\gamma) \qquad \forall\, \omega \in \mathscr{D}^1(\R^n)\,,
\]
where $\llbracket \gamma \rrbracket$ denotes the multiplicity one $1$-current canonically associated with $\gamma \in \Lip$. The measure $\pi$ satisfies the identity $2\M(\pi)=\M(\partial T)$. We can then compute, for every smooth compactly supported test function $\phi:\R^n \to \R$,
\begin{align}
\int_{E} \phi \, |\theta| \,d\Haus^1+\int_{\R^n\setminus E} \phi \,d\mu & =\int_{\R^n} \phi \,d(|\theta| \Haus^1\trace E+\mu)\nonumber\\
& =\int_{\Lip}\left(\int_{E} \phi \mathbbm{1}_{\text{Im}\gamma}\,d\Haus^1+\int_{\R^n\setminus E} \phi \mathbbm{1}_{\text{Im}\gamma}\,d\Haus^1\right) d\pi(\gamma)\label{Life_of_Pi}\\
& =\int_{E}\phi\left(\int_{\Lip} \mathbbm{1}_{\text{Im}\gamma}\,d\pi(\gamma) \right)\,d\Haus^1+\int_{\R^n\setminus E}\phi \,d\nu\,,\nonumber
\end{align}
where $\nu$ is a measure supported on $\R^n\setminus E$. The equality implies that 
$$|\theta(x)|=\int_{\Lip} \mathbbm{1}_{\text{Im}\gamma}(x)\,d\pi (\gamma) \leq \M(\pi) = \frac12 \M(\partial T) \qquad \mbox{for $\Ha^1$-a.e. $x \in E$\,,}$$
which concludes the proof.
\end{proof}

\begin{proof}[Proof of Theorem \ref{t:existence}]
In this proof we will implicitly identify vector-valued measures $T\in\mathscr{M}(\R^n,\R^{n\times m})$ and $\mu\in\mathscr{M}(\R^n,\R^m)$ with 1-dimensional and 0-dimensional currents with coefficients in $\R^m$ and finite mass, respectively, and thus we will write $T = A_T \|T\|$ and $\mu = \vec{\mu} \|\mu\|$ for a $\|T\|$-measurable tensor field $A_T \colon \R^n \to \Lambda_1(\R^n) \otimes \R^m \simeq \R^{n \times m}$ and a $\|\mu\|$-measurable vector field $\vec{\mu} \colon \R^n \to \R^m$ . Similarly, recalling the notation of Remark \ref{interpr}, for $j=1,\dots,m$, a component $T_j$ of $T$ and a component $\mu_j$ of $\mu$ are identified respectively with a (classical) 1-dimensional and 0-dimensional current of finite mass.

Let $T$ be a multi-material flux between $\mu^-$ and $\mu^+$, whose components are $(T_1,\dots,T_m)$. For each component $T_j$, let $T'_j$ denote the acyclic part of $T_j$. By the definition  of acyclic part of a current, $T'_j$ is a piece of $T_j$ for every $j$, and thus the current $T' \in \mathscr{F}_1^{\R^m}(\R^n)$ defined in components by $T' = \left( T'_1,\dots,T'_m \right)$ is a piece of $T$. From Proposition \ref{prop:monotonicity} it follows that 

\begin{equation}\label{e:menoenergia}
\E(T')\leq \E(T)\,.
\end{equation}

Furthermore, $\partial T' = \partial T$, and thus $T'$ is also a multi-material flux between $\mu^-$ and $\mu^+$.

Next, we claim that there exists a constant $C>0$ such that
\begin{equation} \label{energy controls mass}
\mass(T') \leq C \en(T)\,.
\end{equation} 

In order to see this, let us write, recalling Notation \ref{defnot:lift},
\[
T = \sum_{j} \hat{T}_j\,, \qquad T' = \sum_j \widehat{T'}_j\,.
\]

Let us decompose each $T_j$ according to Theorem \ref{dec} into its rectifiable and diffuse part:
\[
T_j = \xi_j\, \theta_j \Ha^1 \trace E_j + \tau_j\,\mu_j\,,
\]
where $E_j \subset \R^n$ is $1$-rectifiable, $|\xi_j(x)|=1$ at $\Ha^1$-a.e. $x\in E_j$, $|\tau_j(x)|=1$ $\mu_j$-almost everywhere, and $\mu_j(E) = 0$ for every $1$-rectifiable subset $E$. By definition of acyclic part of a current, we can represent $T'_j$ as
\[
T'_j = \xi_j\, \theta'_j \Ha^1 \trace E_j + \tau_j \lambda_j\, \mu_j\,,
\] 
for some measurable $\lambda_j \colon \R^n \to \left[0,1\right]$ and for $\theta'_j(x) \leq \theta_j(x)$ at $\Ha^1$-a.e. $x \in E$. Also, correspondingly we have the representation
\[
\hat{T}_j = (\xi_j \otimes \theta_j e_j)\, \Ha^1 \trace E_j + (\tau_j \otimes e_j) \, \mu_j\,, \qquad \widehat{T'}_j = (\xi_j \otimes \theta'_j e_j)\, \Ha^1 \trace E_j + \lambda_j (\tau_j \otimes e_j)\, \mu_j\,.
\]
Then, from Lemma \ref{redu} it immediately follows that
\begin{equation} \label{reduction classical 1}
\en(\widehat{T'}_j) = \inf\left\lbrace \liminf_{h \to \infty} \en(P_h) \, \colon \, \{P_h\}_h \mbox{ sequences in } \mathcal{P}_j \mbox{ with } \flat(\widehat{T'}_j - P_h) \to 0 \right\rbrace\,,
\end{equation}
where $\mathcal{P}_j$ is the class of polyhedral $P \in \mathscr{P}_1^{\R^m}(\R^n)$ of the form
\[
P = \sum_{\ell=1}^N (\tau_\ell \otimes e_j) \theta_\ell \, \Ha^1 \trace \sigma_\ell\,,
\] 
where $\sigma_\ell$ are finite unions of segments with disjoint relative interiors.
In turn, the quantity on the right-hand side of \eqref{reduction classical 1} is equivalent to
\[
\inf\left\lbrace \liminf_{h\to\infty} \en_j(P_h) \, \colon \, \{P_h\}_h \mbox{ sequences in } \mathscr{P}_1(\R^n) \mbox{ with } \flat(T'_j - P_h) \to 0 \right\rbrace\,,
\]
where
\[
\en_j(P) := \sum_{\ell=1}^N \cost(\theta_\ell \, e_j) \, \Ha^1(\sigma_\ell) \qquad \mbox{for } P=\sum_{\ell=1}^N \tau_\ell\,\theta_\ell \Ha^1\trace \sigma_\ell\,.
\]
Hence, by \cite[Proposition 2.32]{Brancolini2017} we can explicitly compute
\begin{equation} \label{explicit_components}
\en(\widehat{T'}_j) = \int_{E_j} \cost(\theta'_j(x) e_j)\, d\Ha^1(x) + \frac{\partial^+\cost}{\partial e_j}(0) \int_{\R^n} \lambda_j(x)\, d\mu_j(x) \,.
\end{equation}

Note that, by the above formula, if $j \in \{1,\dots,m\}$ is such that $\frac{\partial^+\cost}{\partial e_j}(0) = \infty$ (namely, if $e_j \notin \mathcal{B}$, using the notation of Lemma \ref{bigteo}) then $\mu_j = 0$ and $T'_j = (T'_j)_{\rec}$.

In particular,
\begin{equation} \label{stimamassa1}
\mass(\lambda_j (\tau_j \otimes e_j) \mu_j) = \int_{\R^n} \lambda_j(x) \, d\mu_j(x) = \left(\frac{\partial^+\cost}{\partial e_j}(0)\right)^{-1} \en(\lambda_j (\tau_j \otimes e_j) \mu_j)\,,
\end{equation}
with the formula valid also when $e_j \notin \mathcal{B}$ if we use the convention that $\infty^{-1} = 0$.

At the same time, from Lemma \ref{l:bound_rect} we deduce that the ratio $|\theta'_j(x)|/\cost(\theta'_j(x)e_j)$ can be bounded by $\max_{|\theta|\le \mass(\mu^--\mu^+)} |\theta|/\cost(\theta)$ for almost every $x\in E$. Hence
\begin{align}
\mass((\xi_j \otimes \theta'_j e_j)\Haus^1\trace E_j) & =\int_{E_j} |\theta'_j|\,d\Haus^1\nonumber\\
& \leq\left( \max_{\{\theta\in\R^m\,:\,|\theta|\leq\mass(\mu^--\mu^+)\}}\frac{|\theta|}{\cost(\theta)}\right)\E((\xi_j \otimes \theta'_je_j)\Haus^1\trace E_j)\,. \label{stimamassa2}
\end{align}

Combining \eqref{stimamassa1} and \eqref{stimamassa2}, and using \eqref{norm-cost-ratio-bound} together with \eqref{spezzamento_bello}, we get that
$$\mass(\widehat{T'}_j)\leq C \E(\widehat{T'}_j)\,,$$
where the constant $C$ depends only on $\cost$ and the quantity $\mass(\mu^--\mu^+)$. Finally, we conclude from \eqref{energia componenti comparabile} and \eqref{e:menoenergia}:
\[
\mass(T') \leq \sum_{j} \mass(\widehat{T'}_j) \leq C \E(T') \leq C \en(T)\,.
\]

By the discussion at the beginning of this Section, this implies that one can choose a minimizing sequence $\{T'_{h}\}_{h\in\N}$ for $\E$ which is precompact (with respect to the topology induced by the flat norm), hence the multi-material transport problem admits a minimizer.
\end{proof}

\section{Existence of multi-material fluxes with finite energy and stability}\label{s:finite_energy}
The aim of this section is to identify a class of multi-material transportation costs $\cost:\R^m\to\R$ (that we call \emph{admissible}) having the property that for any pair of compatible vector-valued measures $\mu^-,\mu^+ \in\mathscr{M}(\R^n,\R^m)$ there exists a multi-material flux $T\in\mathscr{M}(\R^n,\R^{n\times m})$ between $\mu^-$ and $\mu^+$ with $\E(T)<\infty$. We follow the strategy presented in \cite{Brancolini2017}. We deduce a stability result for the multi-material transport problem associated to admissible multi-material transportation costs.

\begin{defn}\label{def:admiss} A multi-material transportation cost $\cost:\R^m\to\R$ is \emph{admissible} in $\R^n$ if there exists a concave, non-decreasing function $\beta:[0,\infty)\to[0,\infty)$ such that $\cost(x,\ldots,x)\leq \beta(x)$ for every $x\in[0,\infty)$ and moreover
\begin{equation}\label{admiss}
\int_0^1\frac{\beta(x)}{x^{2-\frac 1n}}\,dx<+\infty\,.
\end{equation}
\end{defn}

\begin{remark} \label{rmk:properties of beta}
Observe that the validity of \eqref{admiss} implies that $\beta(0)=\lim_{x \to 0^+} \beta(x)=0$. In turn, $\beta(0)=0$, together with the concavity of $\beta$, readily implies that for every $a \geq 0$
\[
\beta (ab) \leq \beta(a) \, b \qquad \mbox{whenever $b \geq 1$}\,.
\] 
\end{remark}

Given a function $\beta:[0,\infty)\to[0,\infty)$, we define, for all $k\in\N$ and $n=1,2,\dots$
$$S^\beta(n,k):=2^{(n-1)k}\beta(2^{-nk})\quad{\mbox{ and }}\quad S^\beta(n):=\sum_{k=1}^\infty S^\beta(n,k).$$
We have the following lemma (see \cite[Lemma 2.15]{Brancolini2017})
\begin{lem} \label{lem:serie geometrica}
Let $\beta:[0,\infty)\to[0,\infty)$ satisfy \eqref{admiss}. Then $S^\beta(n)<\infty$.
\end{lem}

\begin{prop}\label{prop:ex_finite}
Let $\cost:\R^n\to\R$ be an admissible multi-material transport cost. Let $\mu^-,\mu^+\in \mathscr{M}(\R^n,\R^m)$ be a pair of compatible measures with compact support. Then, there exists a multi-material flux $T$ between $\mu^-$ and $\mu^+$ with $\E(T)<\infty$.
\end{prop}

Before proving Proposition \ref{prop:ex_finite}, we need to introduce the following notation. For $x\in\R^n$ and $d>0$ we denote by $Q_d(x)\subset \R^n$ the cube centered at $x$ with diameter $d$, and faces parallel to the coordinate hyperplanes, henceforth called a \emph{coordinate cube}. Given a coordinate cube $Q$ and a number $k\in\N$, we denote $$\Lambda(Q,k):=\{Q^\ell\}_{\ell=1}^{2^{kn}}$$ the collection of the $2^{kn}$ cubes obtained dividing each edge of $Q$ into $2^k$ subintervals of equal length. We denote by $$S(Q,k):=\bigcup_{\ell=1}^{2^{kn}}\partial Q^\ell$$ the $(n-1)$-skeleton of the grid $\Lambda(Q,k)$.

\begin{lem}\label{lemmagriglia}
Let $Q'\subset\R^n$ be a coordinate cube. Let $\{\mu_h\}_{h\in\N \cup \{\infty\}}\subset \mathscr{M}_+(\R^n)$ be a countable family of positive measures supported on $Q'$. Then there exists a coordinate cube $Q\supset Q'$ such that 
\begin{equation}\label{griglia}
\mu_h(S(Q,k))=0, \qquad \forall\, k\in\N\,, \quad \forall \, h \in \N \cup \{\infty \} .
\end{equation}
\end{lem}
\begin{proof}
Since the statement concerns only sets with measure zero, we can assume that $\Mass(\mu_h)=1$ for every $h$. Denote $\mu:=\mu_{\infty}+\sum_{h\in\N}2^{-h}\mu_h$. Let $Q''$ be a coordinate cube such that ${\rm dist}(Q',(\R^n\setminus Q''))\geq 1$. We can assume that the edge length of $Q''$ is an integer number. For every $i=1,\dots,n$ and $k\in\N$ we denote 
$H_{i,k}$ the union of $2^k+1$ hyperplanes, orthogonal to $e_i$, intersecting $Q''$ and partitioning it into $2^k$ slabs of equal volume. Denote also 
$$L_i:=\bigcup_{k\in\N}H_{i,k}.$$
Since $L_i+r\, e_i$ is disjoint from $L_i+s\, e_i$ whenever $r-s\in \R\setminus\Q$, then for every $i$ there exists $\rho_i\in[0,1]$ such that 
$$\mu(L_i+\rho_i\,e_i)=0.$$
We conclude that $Q:=Q''+\sum_i\rho_i\,e_i$ contains $Q'$ and yields \eqref{griglia}.\qedhere

\end{proof}

\begin{proof}[Proof of Proposition \ref{prop:ex_finite}]
Let us denote, as in Remark \ref{interpr}, $\nu:=\mu^+-\mu^-$ and $\nu_j$ its components, for $j=1,\dots,m$. We also denote, for every $j$, $(\nu_j)_-$ and $(\nu_j)_+$ respectively the negative and the positive part of $\nu_j$, and finally we let $\nu_-$ and $\nu_+$ be the vector-valued measures whose components are respectively the $(\nu_j)_-$'s and the $(\nu_j)_+$'s. Consider a coordinate cube $Q$ obtained by Lemma \ref{lemmagriglia} applied to the finite family of measures $\{(\nu_j)_\pm\}$.

For every $k\in\N$ we consider the discrete approximation $\sigma^k_\pm$ of  $\nu_\pm$ subject to the grid $\Lambda(Q,k)$, namely $$\sigma^k_\pm:=\sum_{\ell=1}^{2^{kn}}\theta^\pm_\ell\delta_{x^\ell},$$ where $x^\ell$ are the centres of the cubes $Q^\ell\in\Lambda(Q,k)$ and $\theta^\pm_\ell:={\nu_\pm(Q^\ell)}\in\R^m$. The core of the proof is the estimate of the energy for the simplest possible (discrete) multi-material flux between $\sigma^k_\pm$ and $\sigma^{k+1}_\pm$. 

For $\ell=1,\dots,2^{kn}$ we consider the cones $C^\ell_\pm$ over $\sigma^{k+1}_\pm\trace Q^\ell$ with vertex $x^\ell$ (see \eqref{e:cone}). Denoting $$T^k_\pm:=\sum_{\ell=1}^{2^{kn}}C^\ell_\pm,$$
we observe that $${\rm{div}}(T^k_\pm)=\sigma^{k}_\pm-\sigma^{k+1}_\pm,$$
hence $T^k_\pm$ is a discrete multi-material flux between $\sigma^k_\pm$ and $\sigma^{k+1}_\pm$. 

Consider a cube $Q^\ell\in\Lambda(Q,k)$ and the cubes $Q^{\ell, i}\in\Lambda(Q,k+1)$ $(i=1,\dots,2^n)$ generated by $Q^\ell$, namely those cubes in the grid $\Lambda(Q,k+1)$ that are contained in $Q^\ell$. Denoting $x^{\ell,i}$ the centers of the $Q^{\ell, i}$'s, and $\theta^\pm_{\ell,i}$ the corresponding multiplicities, the energy of $T^k_\pm$ is exactly
$$\E(T^k_\pm)=2^{-k-2}{\rm diam}(Q)\sum_{\ell=1}^{2^{kn}}\sum_{i=1}^{2^n}\cost(\theta^\pm_{\ell,i})\,.$$
Similarly, denoting $(T^k_\pm)_j$ the components of $T^k_\pm$, we have
$$ \E(\widehat{(T^k_\pm)}_j)=2^{-k-2}{\rm diam}(Q) \sum_{\ell=1}^{2^{kn}}\sum_{i=1}^{2^n}\cost((\theta^\pm_{\ell,i})_j\,e_j)\,,$$
where $\theta_j$ is the $j$\textsuperscript{th} component of $\theta$.
By the definition of admissible multi-material transport cost, and using that $\cost$ is non-decreasing with respect to the order relation $\preceq$ in $\R^m$, we have, for $j=1,\dots,m$,
\begin{equation}\label{stimafinale}
\E(\widehat{(T^k_\pm)}_j)\leq 2^{-k-2}{\rm diam}(Q)\sum_{\ell=1}^{2^{kn}}\sum_{i=1}^{2^n}\beta((\theta^\pm_{\ell,i})_j)\,.
\end{equation}
Moreover, since 
$$\sum_{\ell=1}^{2^{kn}}\sum_{i=1}^{2^n}(\theta^\pm_{\ell,i})_j=(\nu_j)_\pm(Q),$$ 
then, by concavity of $\beta$, we have
$$\sum_{\ell=1}^{2^{kn}}\sum_{i=1}^{2^n}\beta((\theta^\pm_{\ell,i})_j)=2^{(k+1)n}\sum_{\ell=1}^{2^{kn}}\sum_{i=1}^{2^n}2^{-(k+1)n}\beta((\theta^\pm_{\ell,i})_j)\leq 2^{(k+1)n}\beta(2^{-(k+1)n}(\nu_j)_\pm(Q)).$$
Therefore we deduce from \eqref{stimafinale} that
\begin{align*}
\E(\widehat{(T^k_\pm)}_j) & \leq \frac{1}{2}{\rm diam}(Q)2^{(k+1)(n-1)}\beta(2^{-(k+1)n}(\nu_j)_\pm(Q))\\
& =\frac{1}{2}{\rm diam}(Q)2^{(k+1)(n-1)}\beta(2^{-(k+1)n})\frac{\beta(2^{-(k+1)n}(\nu_j)_\pm(Q))}{\beta(2^{-(k+1)n})}\\
& =\frac{1}{2}{\rm diam}(Q)S^\beta(n,k+1)\frac{\beta(2^{-(k+1)n}(\nu_j)_\pm(Q))}{\beta(2^{-(k+1)n})}\,.
\end{align*}
Let us write $$K(j,k):=\frac{\beta(2^{-(k+1)n}(\nu_j)_\pm(Q))}{\beta(2^{-(k+1)n})}.$$ 
If $(\nu_j)_\pm(Q)\geq 1$, we have $K(j,k)\leq (\nu_j)_\pm(Q)$ for every $k$ by Remark \ref{rmk:properties of beta}; otherwise, by monotonicity of $\beta$, we have $K(j,k)\leq 1$ for every $k$. Summing over $j=1,\dots,m$, we conclude from Proposition \ref{lem_mon} $(2)$ that, for every $N\leq M\in\N$, it holds
\begin{equation}\label{stima_eeeneeergeticaaa}
\E\left(\sum_{k=N}^M T^k_\pm\right)\leq \sum_{j=1}^m\E\left(\sum_{k=N}^M \widehat{(T^k_\pm)}_j\right)\leq\frac{m}{2}{\rm diam}(Q)\left(\sum_{k=N}^MS^\beta(n,k+1)\right)\max\{1,\M(\nu_\pm)\}.
\end{equation}
By Lemma \ref{bigteo} (6), and using that the multiplicities of $T^k_{\pm}$ are bounded by construction, a similar estimate holds for the mass of $\sum_{k=N}^M (T^k_\pm)$. In particular, the sequence $\{S^M_\pm:=\sum_{k=0}^M (T^k_\pm)\}_{M\in\N}$ is Cauchy in mass (notice that $\lim_{k \to \infty} S^\beta (n,k) = 0$ by Lemma \ref{lem:serie geometrica}), and thus it converges in mass to a multi-material flux $S_\pm$. Moreover, $T:=S_+-S_-$ is a multi-material flux between $\nu_-$ and $\nu_+$, or equivalently between $\mu^-$ and $\mu^+$. By \eqref{stima_eeeneeergeticaaa}, we have
\begin{equation} \label{eq:uniform energy bound}
\E(T)\leq m\max\{1,\M(\nu_\pm)\}{\rm diam}(Q)S^\beta(n) < \infty\,,
\end{equation}
and the proof is complete.
\end{proof}
\subsection{Proof of the stability Theorem \ref{t:stability}}\label{ss:stability}
We will actually prove that every subsequential limit $T_\infty$ is a minimizer for the pair $\left( \mu_{\infty}^-, \mu_{\infty}^+ \right)$. Without loss of generality we can assume that $T_h\wto T_\infty$. Furthermore, as a consequence of \eqref{eq:uniform energy bound} and of the lower semicontinuity of the energy, it holds
\begin{equation} \label{eq:limit energy is finite}
\begin{split}
\E(T_\infty) &\leq \liminf_{h \to \infty} \E(T_h) \leq C(m,{\rm diam}(K))\, S^{\beta}(n)\,\left(1 + \sup_{h} \{\M(\mu_h^{\pm})\}\right) \\
&\leq C(m,n,\beta,{\rm diam}(K),\M(\mu^\pm))\,.
\end{split}
\end{equation}

Towards a contradiction, let us assume that $T_\infty$ is not a minimizer. Then, there exist $\delta>0$ and a multi-material flux $S$ between $\mu^-_\infty$ and $\mu^+_\infty$ which satisfies $$\E(S)\leq \E(T_\infty)-7\delta.$$ We will use $S$ to construct a competitor $S_h$ for $T_h$ ($h$ large enough) such that $\E(S_h)<\E(T_h)$, which is a contradiction.
By the lower semi-continuity of $\E$, there exists $h_0\in\N$ such that for $h>h_0$ it holds 
$$\E(T_\infty)\leq \E(T_h)+\delta.$$ Let $Q$ be a cube obtained by Lemma \ref{lemmagriglia}. By \eqref{stima_eeeneeergeticaaa}, there exists $l\in\N$ such that, for every $h\in\N \cup \{\infty\}$
$$\E\left(\sum_{k=l}^\infty (T_h)^k_\pm\right)\leq \delta,$$
where $\sum_{k=l}^\infty (T_h)^k_\pm=:(T_h^l)_\pm$ is a multi-material flux between the discrete approximation $(\sigma_h)_{\pm}^l$ of $\mu_h^\pm$ (subject to the grid $\Lambda(Q,l)$) and the measure $\mu_h^\pm$. Since $\mu_h^\pm\wto\mu_\infty^\pm$, then, for every $\varepsilon>0$, there exists $h_1\geq h_0$ such that the multiplicity of $(\sigma_h)_\pm^l-(\sigma_\infty)_\pm^l$ has norm less than $\varepsilon$ for every $h\geq h_1$. Since $\cost(\theta)\to 0$ as $\theta\to 0$ (see Remark \ref{rmk:properties of beta}), the smallness of the multiplicities of $(\sigma_h)_\pm^l-(\sigma_\infty)_\pm^l$ implies that the cone $C$ over $$(\sigma_h)_+^l-(\sigma_\infty)_+^l-(\sigma_h)_-^l+(\sigma_\infty)_-^l$$
with vertex in the centre of the cube $Q$ satisfies $\E(C)\leq\delta$ for $\varepsilon$ sufficiently small.
The final contradiction is given by the fact that, for $h\geq h_1$, the vector-valued measure
$$S_h:=S+C+(T_h^l)_+-(T_h^l)_--(T_\infty^l)_++(T_\infty^l)_-$$
is a multi-material flux between $\mu_h^-$ and $\mu_h^+$ and by subadditivity of the energy, it holds 
\[
\pushQED{\qed} 
\E(S_h)\leq \E(S)+5\delta\leq\E(T_\infty)-2\delta\leq \E(T_h)-\delta\,.  \qedhere
\popQED
\]   
\begin{remark}[Metrization property of the minimal energy]
Given two compatible measures $\mu^-,\mu^+\in\mathscr{M}(\R^n,\R^m)$, we denote 
$$W(\mu^-,\mu^+):=\min\{\en(T): \mbox{$T$ is a multi-material flux between $\mu^-$ and $\mu^+$}\}\,.$$ 
A simple byproduct of the proof of Theorem \ref{t:stability} is the following: if $T_h$ are minimizers of the multi-material transport problem converging to $T_\infty$, then necessarily $\E(T_h) \to \E(T_\infty)$. In turn, this implies that, if the multi-material transport cost $\cost$ is admissible, then $W$ metrizes the weak-$^*$ convergence; in other words, if a sequence of pairs of compatible measures $\{\mu_h^-,\mu_h^+\}_{h\in\N}$ satisfies $\mu^\pm_h\wto \mu$ for some measure $\mu\in\mathscr{M}(\R^n,\R^m)$, then $W(\mu^-_h,\mu^+_h)\to 0$ as $h\to\infty$. 
\end{remark}

\section{Chains with coefficients in groups}\label{currents}

The following two final sections are devoted to the proof of the representation formula for the energy of a rectifiable multi-material flux stated in Theorem \ref{t:repr_networks}. The result will be here obtained as a particular case of a more general theorem valid in the context of $k$-dimensional chains with coefficients in a normed Abelian group $\G$; see Theorem \ref{reprect}. Such a result is concerned with the representation, on rectifiable $\G$-chains, of a class of functionals defined on flat $G$-chains via relaxation of corresponding energies defined on polyhedral $\G$-chains by integration of cost functions $\cost$ analogous to that considered in Definition \ref{def:mmtcost}. The representation formula for rectifiable multi-material fluxes simply follows by applying Theorem \ref{reprect} with $k=1$ and $\G = \R^m$.

Before proceeding, we are going to collect in this section the fundamental notions concerning $k$-dimensional chains in $\R^n$ with coefficients in a normed group. For a thorough discussion about this topic, we refer the reader to the seminal paper \cite{Fleming}, as well as to the recent contributions \cite{W_deformation,W_rectifiability,DPH_advanced,DPH_basic}.

\subsection{Polyhedral chains with coefficients in a normed group} \label{ssec:catene poliedrali}
Let $\G = \left( \G, + \right)$ denote an Abelian additive group. A \emph{norm} on $\G$ is any function
\[
\| \cdot \| \colon \G \to \R
\]
satisfying the following properties:
\begin{itemize}
\item[$(i)$] $\| g \| \geq 0$ for every $g \in \G$, and $\| g \| = 0$ if and only if $g = 0 \in \G$;

\item[$(ii)$] $\| - g \| = \| g \|$ for every $g \in \G$;

\item[$(iii)$] $\| g + h \| \leq \| g \| + \| h \|$ for every $g,h \in \G$.
\end{itemize}

We will assume that there is a well defined \emph{norm} $\| \cdot \|$ on $\G$ which makes $\G$ a complete metric space with respect to the canonical distance ${\rm d}(g,h) := \| g - h \|$ for $g,h \in \G$.

Let $K$ be a convex compact subset of $\R^{n}$. If $\sigma \subset K$ is a $k$-dimensional \emph{oriented} simplex, then we denote by $\llbracket \sigma \rrbracket$ the classical integral $k$-current canonically associated with $\sigma$. 

A $k$-dimensional \emph{polyhedral chain with coefficients} in $\G$ (or simply a polyhedral $\G$-chain) is a formal \emph{finite} linear combination
\begin{equation} \label{def:polyhedral}
P = \sum_{\ell=1}^{N} g_{\ell} \llbracket \sigma_\ell \rrbracket
\end{equation}
of non-overlapping oriented $k$-simplexes $\sigma_\ell$ with coefficients $g_{\ell} \in \G$. A \emph{refinement} of $P$ is any $k$-dimensional polyhedral $\G$-chain of the form
\[
\sum_{\ell=1}^{N} \sum_{h=1}^{H_{\ell}} g_{\ell}^h\, \llbracket \sigma_\ell^h \rrbracket
\]
where $\sigma_\ell^h\cup\ldots\cup \sigma_\ell^{H_\ell}=\sigma_\ell$ and $g_{\ell}^h = g_\ell$ if $\sigma_{\ell}^h$ has the same orientation of $\sigma_\ell$ or $g_\ell^h = - g_\ell$ otherwise. Two $k$-dimensional polyhedral $\G$-chains are \emph{equivalent} if they have a common refinement.

The set of $k$-dimensional polyhedral $\G$-chains in $K$ can be given the structure of additive group,  denoted with $\Po_{k}^{\G}(K)$, as follows. The sum of two polyhedral $\G$-chains $P_1$ and $P_2$ is obtained by firstly finding refinements of $P_1$ and $P_2$ such that the corresponding simplexes are either non-overlapping or they coincide, and then taking their formal sum with the identification $g_1\, \llbracket \sigma \rrbracket + g_2 \,\llbracket\sigma\rrbracket = (g_1+g_2)\, \llbracket\sigma\rrbracket$. 

Notice that an element $P \in \Po_{0}^{\G}(K)$ is a $\G$-valued discrete measure of the form $P = \sum_{a \in A} g(a) \llbracket a \rrbracket$: here, $A \subset K$ is a \emph{finite} set, $g(a) \in \G$ for every $a \in A$, and the $\G$-valued measure $g(a) \llbracket a \rrbracket$ is defined by
\[
g(a) \llbracket a \rrbracket(E) :=
\begin{cases}
g(a) &\mbox{ if $a \in E$ } \\
0 &\mbox{ otherwise }\,.
\end{cases}
\] 

If $P$ is as in \eqref{def:polyhedral}, then the \emph{mass} of $P$ is defined by
\begin{equation} \label{def:mass}
\mass(P) := \sum_{\ell=1}^{N} \| g_\ell \| \Ha^{k}(\sigma_\ell)\,.
\end{equation}

\subsection{Rectifiable chains with coefficients in a normed group}
More generally, a $k$-dimensional \emph{Lipschitz $\G$-chain} in $K$ has the form  
\begin{equation} \label{def:lipschitz}
\sum_{\ell=1}^{N} g_{\ell} \cdot (\gamma_{\ell})_{\sharp} \llbracket \sigma_\ell \rrbracket \,,
\end{equation}
where each $\sigma_\ell$ is an oriented $k$-simplex in $\R^{k}$, $g_\ell \in \G$, $\gamma_{\ell} \colon \sigma_\ell \to \R^n$ is Lipschitz with $\gamma_\ell(\sigma_\ell) \subset K$, and $\gamma_{\sharp}$ is the \emph{push-forward} operator associated to the Lipschitz map $\gamma$. Analogous considerations to those made in the definition of $k$-dimensional polyhedral $\G$-chains allow to define the group of $k$-dimensional Lipschitz $\G$-chains in $K$, denoted $\mathscr{L}_{k}^{\G}(K)$, and to extend the mass functional to $\mathscr{L}_{k}^{\G}(K)$.

The $\mass$-completion of $\mathscr{L}_{k}^{\G}(K)$ is the group $\Rc_{k}^{\G}(K)$ of $k$-dimensional \emph{rectifiable chains} with coefficients in $\G$. Observe that an element $R \in \Rc_{0}^{\G}(K)$ is a $\G$-valued atomic measure of the form $R = \sum_{a \in A} g(a) \llbracket a \rrbracket$ for some \emph{countable} $A \subset K$, and $g(a) \in \G$ for every $a \in A$ such that $\mass(R) = \sum_{a \in A} \| g(a) \| < \infty$. The mass $\mass(\cdot)$ is a norm on the group $\Rc_{k}^{\G}(K)$.

Let $V \subset \R^{n}$ be a $k$-dimensional vector subspace. A $\G$-valued \emph{orientation} of $V$ is an equivalence class of pairs $(\tau, g)$, where $\tau\in\Lambda_k(\R^n)$ is a unit mass orientation of $V$ (that is, $\tau = \tau_1 \wedge \ldots \wedge \tau_k$ for an orthonormal basis $\{\tau_1,\ldots,\tau_k\}$ of $V$) and $g \in \G$, defined by the equivalence relation
\[
(\tau,g) \equiv (\xi,h) \quad \mbox{if and only if} \quad (\tau = \xi \mbox{ and } g = h) \mbox{ or } (\tau = -\xi \mbox{ and } g = -h)\,.
\]
We introduce the notation $\tau \otimes g$ for the $\equiv$-equivalence class $\left[ (\tau,g) \right]$, since, despite being non-standard, it is coherent with the one used in the previous sections when $k=1$ and $\G = \R^m$.

If $E \subset \R^n$ is (countably) $k$-rectifiable, then a $\G$-valued orientation of $E$ is a $\Ha^{k}$-measurable choice of an orientation $(\tau \otimes g)(x)$ for the ($\Ha^{k}$-a.e. well defined) approximate tangent spaces ${\rm Tan}(E,x)$. It can be seen that if $R \in \Rc_{k}^{\G}(K)$ then $R$ is associated with a $k$-rectifiable set $E \subset K$ having an $\Ha^{k}$-integrable $\G$-orientation defined on it (see \cite[Section 3.6]{DPH_advanced}). In this case, we shall write $R = \llbracket E, \tau \otimes g \rrbracket$. Furthermore, it holds
\begin{equation} \label{the mass of rectifiable chains}
\mass(R) = \int_{E} \| g \| \, d\Ha^{k} \,.
\end{equation}
\subsection{Boundary and flat norm}\label{subs_flat}
If $P \in \Po_{k}^{\G}(K)$ is a $k$-dimensional polyhedral $\G$-chain of the form
\[
P = \sum_{\ell=1}^{N} g_{\ell} \llbracket \sigma_{\ell} \rrbracket\,,
\]
then the \emph{boundary} of $P$ is the $(k-1)$-dimensional polyhedral $\G$-chain defined by
\begin{equation} \label{bdry_polyhedral}
\partial P := \sum_{\ell=1}^{N} g_{\ell} \partial \llbracket \sigma_{\ell} \rrbracket\,,
\end{equation}
where $\partial \llbracket \sigma \rrbracket$ is the classical boundary of $\llbracket \sigma \rrbracket$ in the sense of integral currents.

Let $P \in \Po_{k}^{\G}(K)$. The \emph{flat norm} of $P$ is
\begin{equation} \label{def:flat_norm}
\flat(P) := \inf\left\lbrace \mass(Q) + \mass(P - \partial Q) \, \colon \, Q \in \Po_{k+1}^{\G}(K) \right\rbrace\,.
\end{equation}

Observe that $\flat(P) \leq \mass(P)$ by definition, and that $\flat(\partial P) \leq \flat(P)$ (note that, as usual, $\partial (\partial Q) = 0$ for every polyhedral $Q$). 
\subsection{Flat $G$-chains}
The $\flat$-completion of $\Po_{k}^{\G}(K)$ is the group $\F_{k}^{\G}(K)$ of $k$-dimensional \emph{flat} $\G$-chains in $K$. The same group of flat $\G$-chain would be obtained by completing the Lipschitz $\G$-chains $\mathscr{L}_{k}^{\G}(K)$ with respect to an analogously defined flat norm. It holds true that $\Rc_{k}^{\G}(K) \subset \F_{k}^{\G}(K)$ with continuous inclusion with respect to the mass topology on $\Rc_{k}^{\G}(K)$ and the flat topology on $\F_{k}^{\G}(K)$. By \cite[Theorem 5.3.1]{DPH_advanced}, if $R \in \Rc_{k}^{\G}(K)$, then its flat norm is given by
\begin{equation} \label{flat_rect}
\flat(R) = \inf\left\lbrace \mass(S) + \mass(Z) \, \colon \, S \in \Rc_{k}^{\G}(K) \mbox{ and } Z \in \Rc_{k+1}^{\G}(K) \mbox{ with } R = S + \partial Z \right\rbrace\,.
\end{equation}

The boundary operator $\partial \colon \Po_{k+1}^{\G}(K) \to \Po_{k}^{\G}(K)$ admits a continuous extension \[\partial \colon \left( \F_{k+1}^{\G}(K), \flat \right) \to \left( \F_{k}^{\G}(K), \flat \right) \] such that $\partial (\partial T) = 0$ and $\flat(\partial T) \leq \flat(T)$ for every $T \in \F_{k+1}^{\G}(K)$. 

\subsection{The case ${\rm G} = \R^m$: comparison with Section \ref{s:syllabus_euclidean}}
The constructions outlined in the previous paragraphs evidently apply as well to the case when ${\rm G} = \R^m$, thus leading to seemingly different definitions of the classes of chains with coefficients in $\R^m$ compared to those given in Section \ref{s:syllabus_euclidean}. It is easily seen that the two approaches are in fact equivalent. This follows directly from the following observations: 
\begin{itemize}
\item[(i)] the classes of polyhedral currents and chains defined, respectively, in Subsections \ref{ssec:correnti rettificabili} and \ref{ssec:catene poliedrali} coincide; analogously, the mass functional, and therefore the flat norm, are defined in the same way on polyhedral currents;
\item[(ii)] rectifiable $k$-currents with coefficients in $\R$ (as defined in Section \ref{s:syllabus_euclidean} with $m=1$) are the $\mass$-completion of $k$-dimensional Lipschitz $\R$-chains by \cite[Theorem 4.1.28]{Federer_GMT}; analogously, flat $k$-currents with coefficients in $\R$ (as defined in Section \ref{s:syllabus_euclidean} with $m=1$) are the $\flat$-completion of $k$-dimensional polyhedral chains by \cite[4.1.23]{Federer_GMT};
\item[(iii)] a $k$-current in $\R^n$ with coefficients in $\R^m$ (as defined in Section \ref{s:syllabus_euclidean}) is polyhedral, rectifiable, flat, or of finite mass if and only if all its components are such.
\end{itemize}
By virtue of these considerations, and as already anticipated, Theorem \ref{t:repr_networks} is just a rewriting of Theorem \ref{reprect} below in the case ${\rm G} = \R^m$ with $k=1$.

\subsection{Restriction and slicing}\label{slic}

We will denote by $R \trace U$ the \emph{restriction} of a rectifiable $R \in \Rc_{k}^{\G}(K)$ to a Borel subset $U$ (cf. \cite[Section 3.4]{DPH_advanced}). In particular, if $R \in \Rc_{0}^{\G}(K)$ has the form
\[
R = \sum_{a \in A} g(a) \llbracket a \rrbracket\,,
\] 
then 
\[
R \trace U = \sum_{a \in A \cap U} g(a) \llbracket a \rrbracket\,.
\]

Recall that if $E \subset \R^{n}$ is $k$-rectifiable, and if $f \colon \R^{n} \to \R^{p}$ is Lipschitz with $p \leq k$, then the set $E \cap f^{-1}(\{y\})$ is $(k-p)$-rectifiable for (Lebesgue) almost every $y \in \R^p$. If $R = \llbracket E, \tau \otimes g \rrbracket \in \Rc_{k}^{\G}(K)$, then for a.e. $y \in \R^p$ it is well defined (see \cite[Section 3.7]{DPH_advanced}) the \emph{slice} of $R$ via $f$ at $y$, denoted
\[
\langle R, f, y \rangle \in \Rc_{k-p}^{\G}(K)\,.
\] 
For these $y$, the rectifiable $\G$-chain $\langle R, f, y \rangle$ has supporting set on $E \cap f^{-1}(\{y\})$, and at $\Ha^{k-p}$-a.e. $x \in E \cap f^{-1}(\{y\})$ the $\G$-orientation of $\langle R, f, y \rangle$ at $x$ is $\pm (\tau \otimes g)(x)$, where the $\pm$ sign is determined depending on the behaviour of $f$ in a neighborhood of $x$.

The following formulae involving the operations just introduced will be very useful in the sequel.
\begin{prop}[{\cite[Theorems 3.7.1 and 5.2.4]{DPH_advanced}}]
Suppose that $S,T \in \Rc_{k}^{\G}(K)$, $U$ is a Borel measurable subset of $\R^n$, and $f \colon \R^{n} \to \R^{p}$ is Lipschitz, with $p \leq k$. Then, the following conclusions hold true:
\begin{equation} \label{sum_restriction}
(S + T) \trace U = S \trace U + T \trace U \,;
\end{equation}
\begin{equation} \label{sum_slicing}
\langle S + T, f, y \rangle = \langle S, f, y \rangle + \langle T, f, y \rangle \quad \mbox{for a.e. } y \in \R^p\,;
\end{equation}
\begin{equation} \label{restr_slicing}
\langle T \trace U, f, y \rangle = \langle T,f,y \rangle \trace U \quad \mbox{for a.e. } y \in \R^p\,;
\end{equation}
\begin{equation} \label{GAC}
\mass(T \trace U) \leq \mass(T)\,;
\end{equation}

\begin{equation} \label{slicing_coarea}
\int_{\R^p} \mass(\langle T, f, y \rangle) \, dy \leq C_{k,p} (\Lip(f))^{p} \mass(T)\,.
\end{equation}

Furthermore, if $p=1$ and $\partial T$ is also rectifiable then one has:
\begin{equation} \label{slicing_formula}
\langle T,f,y \rangle = \partial(T \trace \{ f \leq y \}) - (\partial T) \trace \{ f \leq y \} \quad \mbox{for a.e. } y \in \R\,.
\end{equation}

\end{prop}

We will also need the following result.

\begin{prop}[{\cite[Theorem 2.1]{W_rectifiability}}] \label{augmentami_ti_prego}
There exists a group homomorphism (typically known as the \emph{augmentation map}) $\chi \colon \F_{k}^{\G}(K) \to \G$ with the following properties:
\begin{itemize}
\item[$(i)$] $\chi\left( \sum_{a}g(a) \llbracket a \rrbracket \right) = \sum_{a}g(a)$;

\item[$(ii)$] $\chi(\partial T) = 0$ for every $T \in \F_{1}^{\G}(K)$;

\item[$(iii)$] $\| \chi(T) \| \leq \flat(T)$;

\item[$(iv)$] $\flat(T) \leq \| \chi(T)\| + \mass(T) \diam(\spt(T))$.

\end{itemize}
\end{prop}

\section{The representation theorem on rectifiable $\G$-chains}\label{s:representation}

Let $\left( \G, \| \cdot \| \right)$ be a normed Abelian additive group as above. 

We will consider a \emph{cost function}
\[
\cost \colon \G \to \left[ 0, \infty \right)
\]
satisfying the following properties:
\begin{itemize}

\item[$(C1)$] $\cost$ is \emph{even}, that is $\cost(-g) = \cost(g)$ for every $g \in \G$, and furthermore $\cost(g) = 0$ if and only if $g = 0 \in \G$;

\item[$(C2)$] $\cost$ is \emph{lower semi-continuous}, namely
\[
\cost(g) \leq \liminf_{h \to \infty} \cost(g_h)\,,
\]
whenever $\{ g_h \}_{h=1}^{\infty}$ is a sequence in $\G$ such that $\| g - g_h \| \to 0$ as $h \uparrow \infty$;

\item[$(C3)$] $\cost$ is \emph{subadditive}, that is
\[
\cost(g_1 + g_2) \leq \cost(g_1) + \cost(g_2) \quad \mbox{for every } g_1, g_2 \in \G\,.
\]

\end{itemize}

Observe that, when $\G = \R^m$, any cost function as above which, in addition, is monotone non-decreasing is a multi-material transportation cost as in Definition \ref{def:mmtcost}.

Let now $K \subset \R^n$ be a convex compact set. If $P \in \Po_{k}^{\G}(K)$ is a $k$-dimensional polyhedral $\G$-chain of the form
\[
P = \sum_{\ell=1}^{N} g_{\ell} \llbracket \sigma_\ell \rrbracket
\]
for some $g_\ell \in \G$ and $\sigma_\ell$ non-overlapping oriented $k$-simplexes, then we can define the \emph{energy} of $P$ by setting
\begin{equation} \label{def_phi}
\en(P) := \sum_{\ell=1}^{N} \cost(g_\ell) \Ha^{k}(\sigma_\ell)\,.
\end{equation}
Observe that $\en(P) = \mass(P)$ with the choice $\cost(g) = \|g\|$.

This definition naturally extends via relaxation to any $k$-dimensional flat $\G$-chain $T$, thus allowing to define the functional
\[
\E \colon \F_{k}^{\G}(K) \to \R
\]
by setting
\begin{equation} \label{def_E}
\E(T) := \inf\left\lbrace \liminf_{h \to \infty} \en(P_{h}) \, \colon \, \{P_{h}\} \subset \Po_{k}^{\G}(K) \mbox{ with } \flat(T - P_h) \to 0 \right\rbrace\,.
\end{equation}

The following theorem is the anticipated result concerning the representation of $\en(T)$ when $T \in \Rc_{k}^{\G}(K)$.

\begin{theorem} \label{reprect}

Let $R \in \Rc_{k}^{\G}(K)$. If $R = \llbracket \Sigma, \tau \otimes g \rrbracket$ is associated with the $k$-rectifiable set $\Sigma$ and the $\G$-valued orientation $\tau \otimes g$, then 
\begin{equation}
\E(R) = \int_{\Sigma} \cost(g(x)) \, d\Ha^{k}(x)\,.
\end{equation}

\end{theorem}

\begin{notazioni}

From now on, if $R = \llbracket \Sigma, \tau \otimes g \rrbracket \in \Rc_{k}^{\G}(K)$, we shall set
\begin{equation} \label{def_E0}
\E_{0}(R) := \int_{\Sigma} \cost(g(x)) \, d\Ha^{k}(x)\,.
\end{equation}

\end{notazioni}

\begin{remark}

Note that, by property $(C1)$, the energy $\E_{0}$ is well defined on $\Rc_{k}^{\G}(K)$, in the sense that the integrand only depends on the $\G$-orientation $\tau \otimes g := \left[ (\tau,g) \right]_{\equiv}$ of $R$, and not on the specific representative $(\tau,g)$.

\end{remark}

The first step towards a proof of Theorem \ref{reprect} consists of showing that the energy $\E_0$ is lower semi-continuous with respect to flat convergence of rectifiable $\G$-chains. 

\begin{prop}[Lower semi-continuity of $\E_0$] \label{prop:semi-continuity}

Let $A \subset \R^{n}$ be an open set. Let $R_{h},R \in \Rc_{k}^{\G}(K)$ be such that $\flat(R - R_h) \to 0$ as $h \uparrow \infty$. Then
\begin{equation} \label{semi-continuity}
\E_{0}(R \trace A) \leq \liminf_{h\to \infty} \E_{0}(R_h \trace A)\,.
\end{equation}

\end{prop}

We are going to need the following result, which extends a formula typically known in the literature as \emph{integral-geometric identity}. To state it, we will make use of the following notation. With ${\rm Gr}(n,k)$ we denote the Grassmannian of linear $k$-dimensional subspaces of $\R^n$. The Haar measure on ${\rm Gr}(n,k)$ is denoted $\gamma_{n,k}$: recall that $\gamma_{n,k}({\rm Gr}(n,k)) = 1$. Finally, if $V \in {\rm Gr}(n,k)$ then $\p_V \colon \R^n \to V$ denotes orthogonal projection onto $V$.

\begin{lem}[Integral-geometric identity]\label{maestosa-integral-geo}
There exists a constant $c = c(n,k)$ such that for any $R \in \Rc_k^{\G}(K)$ it holds:

\begin{equation}\label{e:int_geom}
\E_0(R)=c\int_{{\rm Gr}(n,k)\times\R^k}\E_0 \big(\langle R,\p_V,y\rangle \big) d(\gamma_{n,k}\otimes\Ha^k)(V,y).
\end{equation}
\end{lem}
\begin{proof}
The identity is a consequence of \cite[3.2.26; 2.10.15]{Federer_GMT}, which states that if $E \subset \R^n$ is $k$-rectifiable then
\begin{equation}
\label{eqn:standard_igm}
\Ha^k(E) = c\int_{{\rm Gr}(n,k)}\int_{\R^k} \Ha^{0}( \p_{V}^{-1}(\{y\}) \cap E)   \, d\Ha^k(y)\, d\gamma_{n,k}(V).
\end{equation}
for some $c = c(n,k)$. Indeed, for any Borel set $A \subset \R^{n}$, denoting $f = \mathbf{1}_{A}$, \eqref{eqn:standard_igm} implies that
\begin{equation}\label{e:int_geom-f}
\int_{E} f(x)\,d\Ha^k(x) = c\int_{{\rm Gr}(n,k)}\int_{\R^k} \int_{E} f(x) \, \mathbf{1}_{  \p_{V}^{-1}(\{y\})}(x) \, d\Ha^0(x)  \, d\Ha^k(y)\, d\gamma_{n,k}(V).
\end{equation}
Since the previous equality is linear in $f$, it holds also when $f$ is piecewise constant. Since the measure $\Ha^k \trace E$ is $\sigma$-finite, the equality can be extended to any measurable function $f\in L^1(\Ha^{k} \trace E)$. The case $f \notin L^{1}(\Ha^{k} \trace E)$ follows from the Monotone Convergence Theorem via a simple truncation argument.

Taking $R= \llbracket E, \tau \otimes g \rrbracket$, and applying \eqref{e:int_geom-f} with $f(x)= \cost(g(x))$, we deduce that 
$$\E_0(R) =
 c\int_{{\rm Gr}(n,k)}\int_{\R^k} \int_{E\cap  \p_{V}^{-1}(\{y\})} \cost(g(x)) \, d\Ha^0(x)  \, d\Ha^k(y)\, d\gamma_{n,k}(V).
$$
We observe that the right-hand side coincides with the right-hand side in \eqref{e:int_geom},
since, for $\Ha^{k}$-a.e. $y \in \R^{k}$, the $0$-dimensional chain $\langle R,\p_V,y\rangle $ is associated with the set $E \cap \p_{V}^{-1}(y)$ with $\G$-orientation at $\Ha^{0}$-a.e. $x \in E \cap \p_{V}^{-1}(\{y\})$ equal to $\pm (\tau \otimes g) (x)$.
\end{proof}

\begin{proof}[Proof of Proposition \ref{prop:semi-continuity}]
Let us first assume $k=0$. If $R = \llbracket \Sigma, \tau \otimes g \rrbracket$ and $R_{h} = \llbracket \Sigma_{h}, \tau_h \otimes g_h \rrbracket$ then we can formally write
\[
R \trace A = \sum_{x \in \Sigma \cap A} \tau(x) g(x) \llbracket x \rrbracket\,,
\]
and
\[
R_h \trace A = \sum_{x \in \Sigma_{h} \cap A} \tau_h(x) g_h(x) \llbracket x \rrbracket\,,
\]
where $\Sigma = \{x_{i}\}_{i \in \mathbb{N}}$ and $\Sigma_{h} = \{x_{i}^{h}\}_{h \in \mathbb{N}}$ are countable sets, and $\tau,\tau_h \in \{-1,1\}$.

Fix $\e > 0$, and let $N=N(\e) \in \Na$ be such that
\begin{equation} \label{cases}
\begin{split}
\E_{0}(R \trace A) - \sum_{i=1}^{N} \cost(g(x_i)) \leq \e \quad &\mbox{if } \E_{0}(R \trace A) < \infty\,, \\
\sum_{i=1}^{N} \cost(g(x_i)) \geq \frac{1}{\e} \quad &\mbox{otherwise}\,.
\end{split}
\end{equation}

By the properties of the cost functions $\cost$, for every $i \in \{1, \dots, N \}$ there exists a number $\eta_{i} = \eta_{i}(\e,\cost(g(x_i))) > 0$ such that
\begin{equation} \label{cost_lsc}
\cost(g) \geq (1 - \e) \cost(g(x_i))
\end{equation}
whenever $\| g - \tau(x_i) g(x_i) \| \leq \eta_i$. Also, let $0 < r_{i} < \min\{\dist(x_i, \partial A),1 \}$ be such that the balls $B(x_i,r_i)$ are mutually disjoint and moreover
\begin{equation} \label{triangle_first}
\left\| \tau(x_i)g(x_i) - \sum_{x \in \Sigma \cap B(x_i,\rho)} \tau(x) g(x) \right\| \leq \frac{\eta_i}{2} \quad \mbox{ for every } \rho \leq r_i\,.
\end{equation}

Set $\eta_{0} := \min_{1 \leq i \leq N} \eta_i$ and $r_0 := \min_{1 \leq i \leq N} r_i$. We claim the following: there exist $h_{0} \in \Na$ and $\rho_i \in \left( \frac{r_0}{2}, r_0 \right)$ such that
\begin{equation}
\flat((R - R_h) \trace B(x_i,\rho_i)) \leq \frac{\eta_0}{2} \quad \mbox{for every } h \geq h_0\,.
\end{equation}

In order to see this, let $h_0 \in \Na$ be such that 
\[
\flat(R - R_h) \leq \frac{\eta_0 r_0}{16(1 + C_{1,1})} \quad \mbox{ for every } h \geq h_0\,,
\]
where $C_{1,1}$ is the constant from formula \eqref{slicing_coarea}.

Then, by \eqref{flat_rect} there are $S_{h} \in \Rc_{0}^{\G}(K)$ and $Z_{h} \in \Rc_{1}^{\G}(K)$ such that
\[
R - R_{h} = S_{h} + \partial Z_{h}\,, \quad \mbox{ and } \quad \mass(S_{h}) + \mass(Z_{h}) \leq \frac{\eta_0 r_0}{8(1 + C_{1,1})}\,.
\]
  
Observe that also $\partial Z_{h} \in \Rc_{0}^{\G}(K)$. Hence, for a.e. $\rho \in \left( \frac{r_0}{2}, r_0 \right)$ we can use equations \eqref{sum_restriction} and \eqref{slicing_formula} to write
\begin{equation} \label{flat_decomposition}
\begin{split}
(R - R_{h}) \trace B(x_i,\rho) &= S_{h} \trace B(x_i,\rho) + (\partial Z_{h}) \trace B(x_i,\rho) \\
&= S_{h} \trace B(x_i,\rho) - \langle Z_{h}, {\rm d}_{i}, \rho \rangle + \partial (Z_{h} \trace B(x_i,\rho))\,,
\end{split}
\end{equation}
where ${\rm d}_{i}(y) := \abs{y - x_i}$.

Since, by the slicing coarea formula \eqref{slicing_coarea}
\[
\int_{\frac{r_0}{2}}^{r_0} \mass(\langle Z_{h}, {\rm d}_i, \rho \rangle) \, d\rho \leq C_{1,1}\, \mass(Z_{h} \trace (B(x_i,r_0) \setminus B(x_i,\frac{r_0}{2}))) \leq \frac{\eta_0 r_0}{8}\,,
\]
it immediately follows from Fatou's lemma that there exists $\rho_{i} \in \left( \frac{r_0}{2}, r_0 \right)$ such that
\begin{equation}
\liminf_{h \to \infty} \mass(\langle Z_{h}, {\rm d}_i, \rho_i \rangle) \leq \frac{\eta_0}{4}\,,
\end{equation}
and thus \eqref{flat_decomposition} implies that for every $h \geq h_0$, up to possibly passing to a subsequence,
\begin{equation} \label{equazione 915}
\begin{split}
\flat((R-R_h) \trace B(x_i,\rho_i)) &\leq \mass(S_{h} \trace B(x_i,\rho_i)) + \mass(\langle Z_{h}, {\rm d}_i, \rho_i \rangle) + \mass(Z_{h} \trace B(x_i,\rho_i)) \\ &\leq \frac{\eta_0 r_0}{8} + \frac{\eta_0}{4} < \frac{\eta_0}{2}\,.
\end{split}
\end{equation}
Notice that it is possible to take one subsequence such that \eqref{equazione 915} holds for every $i$, since the index $i$ ranges in a finite set.

Invoking Proposition \ref{augmentami_ti_prego}$(iii)$, and denoting $\chi \colon \F_{0}^{\G}(K) \to \G$ the \emph{augmentation homomorphism}, we have that for $h \geq h_0$
\begin{equation} \label{augmentation norm estimate}
\| \chi((R-R_h) \trace B(x_i,\rho_i)) \| \leq \flat((R - R_h) \trace B(x_i,\rho_i)) \leq \frac{\eta_0}{2}\,.
\end{equation} 

On the other hand, by Proposition \ref{augmentami_ti_prego}$(i)$ we also see that
\begin{equation} \label{triangle_second}
\chi((R-R_h) \trace B(x_i,\rho_i)) = \sum_{x \in \Sigma \cap B(x_i,\rho_i)} \tau(x) g(x) - \sum_{x \in \Sigma_h \cap B(x_i,\rho_i)} \tau_{h}(x) g_{h}(x)\,.
\end{equation}

Together, equations \eqref{triangle_first}, \eqref{augmentation norm estimate}, and \eqref{triangle_second} imply that 
\begin{equation}
\left\| \tau(x_i) g(x_i) - \sum_{x \in \Sigma_{h} \cap B(x_i,\rho_i)} \tau_{h}(x) g_{h}(x) \right\| \leq \eta_i \quad \mbox{ for every } h \geq h_0\,.
\end{equation}

By \eqref{cost_lsc} and using that the cost function $\cost$ is even, subadditive and lower semi-continuous (and, thus, countably subadditive), we can therefore conclude that
\begin{equation}
\cost(g(x_i)) \leq \frac{1}{1-\e} \cost\left( \sum_{x \in \Sigma_{h} \cap B(x_i,\rho_i)} \tau_{h}(x) g_{h}(x) \right) \leq \frac{1}{1-\e} \sum_{x \in \Sigma_{h} \cap B(x_i,\rho_i)} \cost(g_{h}(x))\,.
\end{equation}

Summing over $i \in \{1,\dots,N\}$ and using that the balls $B(x_i,\rho_i)$ are pairwise disjoint, we obtain that for every $h \geq h_0$
\begin{equation}
\sum_{i=1}^{N} \cost(g(x_i)) \leq \frac{1}{1-\e} \E_{0}(R_{h} \trace A)\,.
\end{equation}

Passing to the $\liminf$ as $h \uparrow \infty$ and using the fact that $\e$ was arbitrary, this allows to conclude equation \eqref{semi-continuity} when $k=0$ in both cases considered in \eqref{cases}. 

Now, we turn our attention to the case $k \geq 1$. Without loss of generality, let us assume that
\[
\lim_{h \to \infty} \E_{0}(R_{h} \trace A) = \liminf_{h \to \infty} \E_{0}(R_{h} \trace A)\,.
\]

By the slicing coarea formula for the flat norm in \cite[Theorem 5.2.1(4)]{DPH_advanced}, for every plane $V \in {\rm Gr}(n,k)$ one has
\begin{equation} \label{slicing_coarea_flat}
\int_{\R^{k}} \flat(\langle R - R_{h}, \p_V, y \rangle)\, dy \leq C_{k} \flat(R - R_h)\,.
\end{equation}

By integrating equation \eqref{slicing_coarea_flat} in the variable $V \in {\rm Gr}(n,k)$ with respect to the Haar measure $\gamma_{n,k}$ on ${\rm Gr}(n,k)$, and taking the limit as $h \uparrow \infty$, we see that
\begin{equation}
\lim_{h \to \infty} \int_{{\rm Gr}(n,k) \times \R^{k}} \flat(\langle R - R_{h}, \p_V, y \rangle) \, d(\gamma_{n,k} \otimes \Ha^{k})(V,y) = 0\,.
\end{equation}

Thus, along a subsequence (not relabeled), we can conclude that
\begin{equation}
\lim_{h \to \infty} \flat(\langle R - R_{h}, \p_V, y \rangle) = 0 \quad \mbox{for $\gamma_{n,k} \otimes \Ha^{k}$-a.e. $(V,y) \in {\rm Gr}(n,k) \times \R^k$}\,. 
\end{equation}

By \eqref{sum_slicing}, from this it follows that $\langle R_{h}, \p_V,y \rangle$ $\flat$-converges to $\langle R, \p_V, y \rangle$ for $\gamma_{n,k} \otimes \Ha^{k}$-a.e. $(V,y) \in {\rm Gr}(n,k) \times \R^k$. Then, the result for $k = 0$ and \eqref{restr_slicing} yield
\begin{equation}
\E_{0}(\langle R \trace A, \p_V, y \rangle) \leq \liminf_{h \to \infty} \E_{0}(\langle R_{h} \trace A, \p_V, y \rangle) \quad \mbox{for $\gamma_{n,k} \otimes \Ha^{k}$-a.e. $(V,y) \in {\rm Gr}(n,k) \times \R^k$}\,.
\end{equation} 

We conclude the proof by applying twice the integral-geometric identity \eqref{e:int_geom}. Indeed, we easily have:
\begin{equation}
\begin{split}
\E_{0}(R \trace A) &= c \int_{{\rm Gr}(n,k) \times \R^k} \E_{0}(\langle R \trace A, \p_V, y \rangle) \, d(\gamma_{n,k} \otimes \Ha^k)(V,y) \\
&\leq c \int_{{\rm Gr}(n,k) \times \R^k} \liminf_{h \to \infty} \E_{0}(\langle R_h \trace A, \p_V, y \rangle) \, d(\gamma_{n,k} \otimes \Ha^k)(V,y) \\
&\leq \liminf_{h \to \infty} \left( c \int_{{\rm Gr}(n,k) \times \R^k} \E_{0}(\langle R_h \trace A, \p_V, y \rangle) \, d(\gamma_{n,k} \otimes \Ha^k)(V,y) \right) \\
&= \liminf_{h \to \infty} \E_{0}(R_{h} \trace A)\,.
\end{split}
\end{equation} 
\end{proof}

The second ingredient needed for the proof of Theorem \ref{reprect} is the following technical lemma. Here, we shall adopt the following notation. Let $R = \llbracket E, \tau \otimes g \rrbracket$ be a rectifiable $\G$-chain. Let also $x \in E$ be such that ${\rm Tan}(E,x)$ exists. Denote $\pi_{x}$ the affine $k$-plane $\pi_{x} := x + {\rm Tan}(E,x) = x + {\rm span}[\tau(x)]$. Then, for any $r > 0$ we will let $S_{x,r}$ be the rectifiable $\G$-chain defined by $S_{x,r}:= \llbracket \pi_{x} \cap B_{r}(x), \tau(x) \otimes g(x) \rrbracket$: this is the chain supported on the disc $\pi_{x} \cap B_{r}(x)$ with orientation $\tau(x)$ and constant density $g(x) \in \G$. In other words, we may write $S_{x,r} = g(x) \cdot \llbracket \pi_{x} \cap B_{r}(x), \tau(x), 1 \rrbracket$.

\begin{lem}\label{lem:flat_approximation}

Let $R = \llbracket E, \tau \otimes g \rrbracket$ be a $k$-dimensional rectifiable $\G$-chain in $K$, and let $\mu := \| g \|\, \Ha^{k} \trace E$. Then it holds:

\begin{equation} \label{flat_approximation}
\lim_{r \to 0^{+}} \frac{\flat(R \trace B(x,r) - S_{x,r})}{\mass(R \trace B(x,r))} = 0 \qquad \mbox{for $\mu$-a.e. $x$}\,.
\end{equation}

\end{lem}

In the proof of the above lemma, we are going to need a suitable version of the classical Lebesgue points theorem (see e.g. \cite[Corollary 2.23]{AFP}) adapted to the framework of $\G$-valued maps. In fact, the same proof can be extended to the case when the target is an arbitrary (possibly not complete) metric space $(X, {\rm d})$, hence we state the result (Proposition \ref{metric_lebesgue} below) under this more general assumption. We shall need a few preliminaries concerning metric space-valued maps. Let $\Omega \subset \R^n$ be an open set, let $(X, {\rm d})$ be a metric space, and let $f \colon \Omega \to X$. If $\mu$ is a positive finite Borel measure on $\Omega$, and $f$ is (Borel) measurable, we say that $f$ is $\mu$-integrable, and we write $f \in L^1(\Omega, X; \mu)$ provided 
\begin{equation} \label{weak integrability}
\mbox{there exists $p \in X$ such that} \quad \int_{\Omega} {\rm d}(f(x),p) \, d\mu(x) < \infty \,.
\end{equation}
Observe that, since $\mu(\Omega) < \infty$, the condition \eqref{weak integrability} is in fact equivalent to the stronger
\begin{equation} \label{strong integrability}
\int_{\Omega} {\rm d}(f(x),p) \, d\mu(x) < \infty \qquad \mbox{for every $p \in X$}\,.
\end{equation} 

Next, we recall the following elementary fact, known in the literature as \emph{Kuratowski's embedding} (see e.g. \cite{Hainon}): every metric space $(X,{\rm d})$ embeds isometrically into the Banach space $L^\infty(X)$ of bounded functions $\varphi \colon X \to \R$ endowed with the sup-norm $\| \varphi \|_\infty := \sup\left\lbrace |\varphi(p)| \, \colon \, p \in X \right\rbrace$. Such an embedding can be easily obtained by fixing a point $p_0 \in X$ and associating, to every $p \in X$, the function $\varphi_p \in L^\infty(X)$ defined by
\[
\varphi_p(q) := {\rm d}(q,p) - {\rm d}(q,p_0) \qquad \mbox{for every $q \in X$}\,.
\]
Notice that the embedding is not canonical, since it depends on the choice of $p_0$. If $f \colon \Omega \to X$ is as above, and if $\Phi$ denotes a Kuratowski embedding (that is, $\Phi(p)=\varphi_p$ as above), then $F := \Phi \circ f$ maps $\Omega$ into the Banach space $B=L^\infty(X)$, and $f \in L^1(\Omega,X;\mu)$ if and only if
\begin{equation} \label{integrability Banach}
\int_{\Omega} \| F(x) - \Phi(p) \|_{\infty} \, d\mu(x) < \infty \qquad \mbox{for every $p \in X$}\,,
\end{equation}
or, equivalently, if and only if
\begin{equation} \label{classical integrability Banach}
\int_{\Omega} \|F(x)\|_\infty \, d\mu(x) < \infty\,.
\end{equation}

Finally, we recall a few notions concerning Banach space-valued maps. If $(B, \| \cdot \|_B)$ is a (real) Banach space, a map $s \colon \Omega \to B$ is \emph{simple} if there exist $N \in \mathbb{N}$, Borel sets $E_{1}, \ldots, E_N \subset \Omega$, and $\varphi_1, \ldots, \varphi_N \in B$ such that
\[
s(x) = \sum_{i=1}^N \varphi_i \, {\rm 1}_{E_i}(x) \qquad \forall\, x\in \Omega\,,
\] 
where ${\rm 1}_E$ is the indicator function of $E$. A function $F \colon \Omega \to B$ is called:
\begin{itemize}

\item \emph{weakly $\mu$-measurable} if for every $\varphi^* \in B^*$ the (real-valued) function $x \in \Omega \mapsto \langle \varphi^* , F(x) \rangle$ is measurable;

\item \emph{strongly $\mu$-measurable} if there exists a sequence $s_\ell$ of simple functions $s_\ell \colon \Omega \to B$ such that $\lim_{\ell \to \infty} \| F(x) - s_\ell(x)\|_B = 0$ for $\mu$-a.e. $x \in \Omega$;

\item \emph{almost separably-valued} if there exists a set $Z_0 \subset \Omega$ with $\mu(Z_0)=0$ such that $F(\Omega \setminus Z_0) = \{F(x) \, \colon \, x \in \Omega \setminus Z_0\} \subset B$ is separable.

\end{itemize}

The following theorem, known in the literature as Pettis' measurability theorem, establishes the fundamental relationship between the three notions just introduced, providing a necessary and sufficient condition for a Banach space-valued function $F$ to be strongly $\mu$-measurable.

\begin{theorem}[Pettis' measurability theorem, see {\cite[Theorem 2 in Chapter II]{DiestelUhl}}]
A map $F \colon \Omega \to B$ is strongly $\mu$-measurable if and only if it is weakly $\mu$-measurable and almost separably-valued.
\end{theorem}

We are now ready to state and prove the anticipated Lebesgue point theorem for metric space-valued maps.

\begin{prop}[Lebesgue points theorem for metric space-valued maps] \label{metric_lebesgue}

Let $\mu$ be a positive finite Borel measure in an open set $\Omega \subset \R^{n}$, let $\left(X, {\rm d} \right)$ be a metric space, let $\Phi$ be a Kuratowski embedding of $X$ into $L^\infty(X)$, and let $f \in L^{1}(\Omega,X;\mu)$ be a $\mu$-integrable function such that $F := \Phi \circ f$ is strongly $\mu$-measurable. Then, for $\mu$-a.e. point $x$ the following holds:
\begin{equation} \label{eq:metric Lebesgue}
\lim_{r \to 0} \frac{1}{\mu(B(x,r))} \int_{B(x,r)} {\rm d}(f(y), f(x)) \, d\mu(y) = 0\,.
\end{equation}

\end{prop}

\begin{proof}

Since $\Phi$ is an isometry, the conclusion, formula \eqref{eq:metric Lebesgue}, is equivalent to 

\begin{equation} \label{eq: banach Lebesgue}
\lim_{r \to 0} \frac{1}{\mu(B(x,r))} \int_{B(x,r)} \| F(y) - F(x) \|_\infty \, d\mu(y) = 0 \qquad \mbox{for $\mu$-a.e. $x \in \Omega$}\,.
\end{equation}

Since $F$ is strongly $\mu$-measurable, we can apply Pettis' measurability theorem to find a set $Z_0 \subset \Omega$ with $\mu(Z_0) = 0$ such that $Y := F(\Omega \setminus Z_0) \subset B=L^\infty(X)$ is separable. Let $\{ \varphi_{i} \}_{i \in \Na}$ be a dense set in $Y$. For every $i \in \Na$, consider the function $x \in \Omega \setminus Z_0 \mapsto \|F(x) - \varphi_i\|_\infty$. Since this function is $\mu$-integrable because $f \in L^1(\Omega,X;\mu)$ (cf. \eqref{classical integrability Banach}), by the classical Lebesgue points theorem there exists a set $Z_{i}$ with $\mu(Z_{i}) = 0$ such that whenever $x \in \Omega \setminus \left( Z_0 \cup Z_{i} \right)$ it holds
\begin{equation} \label{eq:standard_lebesgue}
\|F(x) - \varphi_i\|_\infty = \lim_{r \to 0} \frac{1}{\mu(B(x,r))} \int_{B(x,r)} \|F(y) - \varphi_i\|_\infty \, d\mu(y)\,.
\end{equation} 

In particular, setting $Z := Z_0 \, \cup \, \bigcup_{i \in \Na} Z_{i}$ one has that equation \eqref{eq:standard_lebesgue} holds for every $i \in \Na$ whenever $x \in \Omega \setminus Z$. Note that $\mu(Z) = 0$. Fix $\e > 0$. Let $x \in \Omega \setminus Z$, and let $\varphi_{i}$ be such that $\|F(x) - \varphi_i\|_\infty \leq \frac{\e}{2}$. Then, we have by triangle inequality:
\begin{eqnarray*}
0 &\leq & \limsup_{r \to 0^+} \frac{1}{\mu(B(x,r))} \int_{B(x,r)} \|F(y) - F(x)\|_\infty \, d\mu(y)\\
&\leq & \|F(x) - \varphi_i\|_\infty + \limsup_{r \to 0^+} \frac{1}{\mu(B(x,r))} \int_{B(x,r)} \| F(y) - \varphi_i\|_\infty \, d\mu(y) \\
&\overset{\eqref{eq:standard_lebesgue}}{=} & 2 \, \| F(x) - \varphi_i\|_\infty \leq \e\,.
\end{eqnarray*} 

The conclusion, formula \eqref{eq: banach Lebesgue}, readily follows from the arbitrariness of $\e$.
\end{proof}

\begin{proof}[Proof of Lemma \ref{lem:flat_approximation}]

Since $E$ is countably $k$-rectifiable, there exist a set $E_{0}$ with $\Ha^{k}(E_0) = 0$, countably many $k$-dimensional planes $\Pi_{i} \subset \R^n$ and $C^1$ and globally Lipschitz maps $f_{i} \colon \Pi_{i} \to \Pi_{i}^{\perp}$ such that
\begin{equation} \label{rettificabilita}
E \subset E_{0} \cup \bigcup_{i \in \mathbb{N}} {\rm Gr}(f_i)\,.
\end{equation}

Set $\Sigma_{i} := {\rm Gr}(f_i)$ for the graph of $f_i$. For every $x \in \bigcup_{i} \Sigma_i$, let $i(x)$ be the first index $i$ such that $x \in \Sigma_i$. Furthermore, for every $i$, let $g_{i}$ be the $\G$-valued map given by
\[
g_{i}(x) :=
\begin{cases}
g(x) &\mbox{ if $i = i(x)$},\\
0 &\mbox{ otherwise}\,.
\end{cases}
\]

Let us denote $R_{i} := \llbracket E \cap \Sigma_{i}, \tau \otimes g_i \rrbracket$. Observe that without loss of generality we can assume that $\left.\tau\right|_{E \cap \Sigma_{i}}$ coincides with the (continuous) orientation $\xi_{i}$ induced on $\Sigma_{i}$ by the orientation on $\Pi_{i}$ through the map $f_{i}$: indeed, otherwise it suffices to replace $\left.\tau\right|_{E \cap \Sigma_{i}}$ with $\left.\xi_{i}\right|_{E \cap \Sigma_{i}}$ and simply change sign to $g$ on the (measurable) set where $\tau \neq \xi_{i}$.

From the definition of $g_i$ it follows that $R = \sum_{i} R_{i}$ and $\mass(R) = \sum_{i} \mass(R_i)$. Hence, for any fixed $\e > 0$ there exists $N = N(\e) \in \mathbb{N}$ such that
\begin{equation} \label{zavorra}
\sum_{i \geq N+1} \mass(R_i) \leq \e\,.
\end{equation}

Define the set $E' \subset E$ by
\begin{equation} \label{insieme_buono}
\begin{split}
E' := \bigg\lbrace x \in E \cap \bigcup_{i=1}^{N} \Sigma_{i} &\quad \mbox{such that $x$ is a Lebesgue point of $g_{i}$} \\
&\mbox{with respect to $\Ha^{k} \trace \Sigma_{i}$ for every $i \in \{1,\dots, N\}$} \bigg\rbrace\,.
\end{split}
\end{equation}

In other words, $E'$ consists of all points $x \in E \cap \bigcup_{i=1}^{N} \Sigma_{i}$ such that
\[
\lim_{r \to 0} \frac{1}{\Ha^{k}(\Sigma_i \cap B(x,r))} \int_{\Sigma_{i} \cap B(x,r)} \| g_{i}(y) - g_{i}(x) \|\, d\Ha^{k}(y) = 0 \quad \mbox{for $i=1,\dots,N$}\,.
\]

Observe that $\mass(R \trace (E \setminus E')) \leq \e$ because of \eqref{rettificabilita}, \eqref{zavorra} and Proposition \ref{metric_lebesgue} applied with $X = \G$ (endowed with the natural metric ${\rm d}(g,h) = \|g - h\|$), $f=g_i$, and $\mu=\Ha^k \mres (E \cap \Sigma_i)$. Notice that, if $\Phi$ denotes a Kuratowski embedding of ${\rm G}$ into $L^\infty(\G)$ then the map $F=\Phi \circ g_i$ is strongly $\Ha^k \mres (E \cap \Sigma_i)$-measurable. In order to see this, let first $\{I_{i}^j\}_{j=1}^\infty$ be a sequence of $k$-dimensional Lipschitz $\G$-chains such that $\mass(R_i - I_{i}^j) \to 0$ as $j \to \infty$. If $\tau_{i}^j \otimes g_i^j$ are the $\G$-orientations of $I_i^j$, and if we fix a choice of representative of the equivalence class such that $\tau_i^j=\tau_i$ $\Ha^k$-a.e. where they are both defined, then the convergence in mass implies, through \eqref{the mass of rectifiable chains}, that $\|g_i - g_{i}^j\| \to 0$ in $L^1(\Ha^k \mres (E \cap \Sigma_i))$ as $j \to \infty$. Notice that each map $g_i^j$ takes at most countably many distinct values in $\G$, obtained as finite sums of elements in the $\mathbb Z$-orbit of the (finitely many) coefficients in $\G$ appearing in the definition of $I_i^j$. Hence, since $\mass(I_i^j) < \infty$, this in turn implies that each $g_i^j$ is an $L^1$-limit of \emph{simple} functions, so that, in particular, $\|g_i - s_{i}^j\| \to 0$ in $L^1(\Ha^k \mres (E \cap \Sigma_i))$ as $j \to \infty$ for simple functions $s_i^j$. Passing to subsequences, this gives the pointwise convergence, as $j \to \infty$, of $\|g_i(x) - s_i^j(x)\|$ to zero for $\Ha^k$-a.e. $x \in E \cap \Sigma_i$.

\medskip

Let us now set 
\begin{equation} \label{Lip}
L := \max\lbrace \Lip(f_{i}) \, \colon \, i=1,\dots,N\rbrace.
\end{equation}

Fix $i \in \{1, \dots, N \}$. For every $x \in \Sigma_{i}$ there exists $r > 0$ such that whenever $j \in \{1, \dots, N \}$ is such that $\Sigma_{j} \cap B(x,\sqrt{n} r) \neq \emptyset$, then $x \in \Sigma_{j}$.

Observe now that the definition of $E'$ implies that for any $x \in E'$
\begin{equation} \label{density}
\exists \, \lim_{r \to 0} \frac{\mass(R_{j} \trace (\Sigma_{j} \cap B(x,r)))}{\Ha^{k}(\Sigma_{j} \cap B(x,r))} = \|g_{j}(x)\| \quad \mbox{for every $j=1,\dots,N$}\,.
\end{equation}

Now, fix any point $x \in E'$, and consider any index $j \in \{ 1, \dots, N\}$ such that $x \in \Sigma_{j}$. When $j = i(x)$, then $g_{j}(x) = g(x)$, with $\|g(x)\| > 0$. In particular, \eqref{density} implies that when $j = i(x)$ there exists $r > 0$ such that for any $0 < \rho \leq \sqrt{n} r$ 
\begin{equation} \label{density2}
\frac{\mass(R_{j} \trace (\Sigma_{j} \cap B(x,\rho)))}{\Ha^{k}(\Sigma_{j} \cap B(x,\rho))} \geq \frac{\|g_{j}(x)\|}{2} > 0.
\end{equation}
Again by Proposition \ref{metric_lebesgue} applied with $\mu = \Ha^{k} \trace \Sigma_{j}$, $X = \G$, and $f = g_{j}$, there exists a radius $r > 0$ (depending on $x$) such that 
\begin{equation} \label{Lebesgue_pt}
\begin{split}
\int_{\Sigma_{j} \cap B(x,\rho)} \|g_{j}(y) - g_{j}(x)\| \, d\Ha^{k}(y) &\leq \e \frac{\|g_{j}(x)\|}{2} \Ha^{k}(\Sigma_{j} \cap B(x,\rho)) \\
&\leq \e \frac{\mass(R_{j} \trace (\Sigma_{j} \cap B(x,\rho)))}{\Ha^{k}(\Sigma_{j} \cap B(x,\rho))} \Ha^{k}(\Sigma_{j} \cap B(x,\rho)) \\
&\leq \e \mass(R_{j} \trace B(x,\rho)),
\end{split}
\end{equation} 
for every $0 < \rho \leq \sqrt{n} r$.

If, instead, $j \neq i(x)$, then $g_{j}(x) = 0$ and therefore there exists a radius $r > 0$ (depending on $x$) such that for every $0 < \rho \leq \sqrt{n} r$
\begin{equation} \label{altro_caso}
\begin{split}
\int_{\Sigma_{j} \cap B(x,\rho)} \|g_{j}(y)\| \, d\Ha^{k}(y) &\leq \frac{\e \|g_{i(x)}(x)\|}{N(1+L)^{k}} \Ha^{k}(\Sigma_{j} \cap B(x,\rho)) \\
&\leq \frac{\e}{N} \|g_{i(x)}(x)\| \omega_{k} \rho^{k} \\
&\overset{\eqref{density2}}{\leq} 2 (1+L)^k \frac{\e}{N} \mass(R_{i(x)} \trace B(x,\rho)),
\end{split}
\end{equation}
where $\omega_{k}$ denotes the volume of the unit ball in $\R^{k}$. Notice that both the second and third inequalities in \eqref{altro_caso} use the fact that the surfaces $\Sigma_j$ are graphs of $L$-Lipschitz functions (see \eqref{Lip}), so that
\[
\frac{\omega_k \rho^k}{(1+L)^k}\leq\Ha^k(\Sigma_j \cap B(x,\rho)) \leq (1 + L)^k\, \omega_k \rho^k\,.
\]

Now, let $i = i(x)$. By representing $f_i(\Pi_i) \cap B(x,r)$ locally as the graph of a $C^1$ function (still denoted $f_i$) from the $k$-plane tangent to $\Sigma_{i}$ at $x$ (still denoted $\Pi_i$), translating and tilting such a plane, we can assume that $x = 0$,
 $\Pi_{i} = \lbrace x_{k+1} = \dots = x_{n} = 0 \rbrace$, $f_i(x)=0$, and $\nabla f_{i}(x) = 0$. By possibly choosing a smaller radius $r = r(x) > 0$, we may also assume that 
\begin{equation} \label{poco_tilt}
|\nabla f_{i}| \leq \e \quad \mbox{in } \Pi_{i} \cap B(x,r).
\end{equation}

With these conventions, for $\rho$ suitably small the chain $S_{x,\rho}$ reads $S_{x,\rho} = \llbracket B(0,\rho) \cap \Pi_{i}, \tau(0) \otimes g_i(0) \rrbracket$. We let $F_{i} \colon \Pi_{i} \times \Pi_{i}^{\perp} \to \R^{n}$ be given by $F_{i}(z,w) := \left( z, f_{i}(z) \right)$, and we set $\tilde{R}_{i} := (F_{i})_{\sharp} S_{x,\rho} \in \Rc_{k}^{\G}(K)$.  

Observe that $S_{x,\rho} = g_{i}(0) \cdot \tilde{S}_{x,\rho}$, with $\tilde{S}_{x,\rho} := \llbracket B(0,\rho) \cap \Pi_i, \tau(0), 1 \rrbracket$, and $\tilde{R}_i = g_i(0) \cdot (F_i)_{\sharp} \tilde{S}_{x,\rho}$. Thus, by the standard homotopy formula for classical currents (cf. \cite[26.22]{Simon1983}), we deduce that

\begin{equation}
\tilde{R}_{i} - S_{x,\rho} = \partial h_{\sharp}(\llbracket (0,1) \rrbracket \times S_{x, \rho}) - h_{\sharp}(\llbracket (0,1) \rrbracket \times \partial S_{x,\rho})\,,   
\end{equation}
where $h \colon (0,1) \times \Pi_{i} \times \Pi_{i}^{\perp} \to \R^n$ is the affine homotopy defined by $h(t,z,w) := t F_{i}(z,w) + (1-t) (z,0)$.

Hence, if we denote $C(x,\rho) := (B(x,\rho) \cap \Pi_{i}) \times \Pi_{i}^{\perp}$, and assume without loss of generality that $K \supset {\rm spt}(\llbracket (0,1) \rrbracket \times S_{x,\rho})$, we have by \cite[Formula 26.23]{Simon1983}
\begin{eqnarray}
\flat(\tilde{R}_{i} - S_{x,\rho}) &\leq & \mass(h_{\sharp}(\llbracket (0,1) \rrbracket \times \partial S_{x,\rho})) + \mass(h_{\sharp}(\llbracket (0,1) \rrbracket \times S_{x,\rho}))\nonumber\\
&\leq & C \| F_{i} - ({\rm id},0) \|_{L^{\infty}(C(x,\rho))} \left( \mass(S_{x, \rho}) + \mass(\partial S_{x, \rho}) \right) \nonumber\\
&\overset{\eqref{poco_tilt}}{\leq} & C \e \rho \left( \mass(S_{x,\rho}) + \mass(\partial S_{x,\rho}) \right)  \nonumber\\
&\leq & C \e \|g(x)\| \omega_{k} \rho^{k} \nonumber\\
&\leq & C \e \|g(x)\| \Ha^{k}(\Sigma_{i} \cap B(x,\rho)) \nonumber\\
&\overset{\eqref{density2}}{\leq} & C \e \mass(R_{i} \trace B(x,\rho)).\label{stima1}
\end{eqnarray}

Now, recall that we can assume that the orientation $\tau$ coincides on $E \cap \Sigma_{i}$ with the continuous orientation $\xi_{i}$ of $\Sigma_{i}$ induced by the orientation of $\Pi_{i} \times \Pi_{i}^{\perp}$ via $F_{i}$. Hence, the rectifiable chain $\tilde{R}_{i}$ reads $\tilde{R}_{i} = \llbracket \Sigma_{i} \cap C(x,\rho), \tau \otimes g_i(x) \rrbracket = g_{i}(x) \cdot \llbracket \Sigma_{i} \cap C(x,\rho), \tau,1 \rrbracket$ (cf. \cite[27.2]{Simon1983}). Therefore, we can compute:
\begin{eqnarray}
\mass(R_{i} \trace B(x,\rho) - \tilde{R}_{i}) &\leq & \mass(R_{i} \trace B(x,\rho) - \tilde{R}_{i} \trace B(x,\rho)) + \mass(\tilde{R}_{i} \trace (C(x,\rho) \setminus B(x,\rho))) \nonumber\\
&\overset{\eqref{Lebesgue_pt}}{\leq}&  \e \mass(R_{i} \trace B(x,\rho)) + \mass(\tilde{R}_{i} \trace (C(x,\rho) \setminus B(x,\rho))) \nonumber\\
&\overset{\eqref{poco_tilt}}{\leq}& \e \mass(R_{i} \trace B(x,\rho)) + C \e \|g_{i}(x)\| \Ha^{k}(\Sigma_{i} \cap B(x,\rho)) \nonumber\\
&\overset{\eqref{density2}}{\leq}& C \e \mass(R_{i} \trace B(x,\rho)).\label{stima2}
\end{eqnarray}

Hence, we conclude:
\begin{eqnarray} 
\flat(R \trace E' \cap B(x,\rho) - S_{x,\rho})\!\!\!\!\!\!\!\!\!\!\!\!   & \leq & \!\!\!\!\!\! \flat(R_{i(x)} \trace B(x,\rho) - S_{x,\rho}) + \sum_{\underset{j \neq i(x)}{j=1}}^{N} \mass(R_{j} \trace B(x,\rho)) \nonumber\\
&\overset{\eqref{altro_caso}}{\leq}& \!\!\!\!\!\! \flat(R_{i(x)} \trace B(x,\rho) - \tilde{R}_{i} ) + \flat(\tilde{R}_{i} - S_{x,\rho}) + 2 \e \mass(R_{i(x)} \trace B(x,\rho)) \nonumber\\
&\overset{\eqref{stima1}, \eqref{stima2}}{\leq} & \!\!\!\!\!\! C \e \mass(R \trace B(x,\rho)).\label{stima_fin} 
\end{eqnarray}
This proves the following: for every $\e > 0$ there exists a set $E' \subset E$ with $\mass(R \trace (E \setminus E')) \leq \e$ such that for every $x \in E'$ there exists $r = r(x) > 0$ such that for every $0 < \rho \leq r$
\begin{equation} \label{eps_stima_finale}
\frac{\flat(R \trace (E' \cap B(x,\rho)) - S_{x,\rho})}{\mass(R \trace B(x,\rho))} \leq \e\,.
\end{equation}

Now, in order to conclude we iterate \eqref{eps_stima_finale}. In particular, for every $i \in \mathbb{N}$ let us denote $E_{i}$ the set $E'$ corresponding to the choice $\e := 2^{-i-1}$, and let $F_i\subset E_i$ be the set of Lebesgue points of $\mathbf{1}_{E_i}$ (inside $E_i$) with respect to $\mu = \|g\| \Ha^{k} \trace E$.
By \cite[Corollary 2.23]{AFP}, the set $F_i$ equals the set $E_i$ up to a set of $\mu$-measure $0$; moreover, for every $x\in F_i$ and for $\rho$ sufficiently small (possibly depending on $x$) it holds
\begin{equation*}
\begin{split}
\mass(R \trace B(x,\rho)- R \trace (E_i \cap B(x,\rho))) &= 
 \int_{(E\setminus E_i)\cap B(x,\rho)} \|g\| \, d\Ha^k  
 \\&\leq 2^{-i-1} \int_{E\cap B(x,\rho)} \|g\|\, d\Ha^k = 2^{-i-1} \mass(R \trace B(x,\rho)).
\end{split}
\end{equation*}
Hence by \eqref{eps_stima_finale} for every $x \in F_i$ there exists $r_i(x)>0$ such that for every $0<\rho< r_i(x)$
\begin{equation*}
\begin{split}
\flat(R \trace B(x,\rho) - S_{x,\rho})&\leq \mass(R \trace B(x,\rho) - R \trace (E_i \cap B(x,\rho)))+ \flat(R \trace (E_i \cap B(x,\rho)) - S_{x,\rho})\\
&\leq 2^{-i} \mass(R \trace B(x,\rho))
\end{split}
\end{equation*}
and 
$$\mass(R \trace (E \setminus F_i)) \leq 2^{-i-1}.$$
Denoting $F:= \bigcup_{i\in \Na}\bigcap_{j\geq i}F_j$, and observing that $E \setminus F= E \cap F^c = E \cap \bigcap_{i\in \Na}\bigcup_{j\geq i}F^c_j $ is contained in $\bigcup_{j\geq i}F^c_j$ for every $i\in \Na$, we have
$$\mass(R\trace(E\setminus F)) \leq \lim_{i \to \infty}\sum_{j=i}^\infty \mass(R\trace(E\setminus F_{j}))\leq \lim_{i \to \infty}\sum_{j=i}^\infty \frac{1}{2^j}=0$$
and this implies that $\Ha^k(E\setminus F)=0$.
Since every $x \in F$ belongs definitively to every $F_j$ (namely, for every $x \in F$ there exists $i_0(x)\in \Na$ such that $x\in F_i$ for every $i \geq i_0(x)$), we obtain \eqref{flat_approximation}.
\end{proof}

\begin{proof}[Proof of Theorem \ref{reprect}]

First observe that by the well known properties of the lower semi-continuous envelope and by Proposition \ref{semi-continuity} it trivially holds true that
\begin{equation} \label{banale}
\E_{0}(R) \leq \E(R) \quad \mbox{for every $R \in \Rc_{k}^{\G}(K)$}\,.
\end{equation}

We prove the opposite inequality. Let $R = \llbracket E, \tau \otimes g \rrbracket \in \Rc_{k}^{\G}(K)$ be a rectifiable $\G$-chain. Starting from $R$, we will construct a sequence $P_{h}$ of polyhedral $\G$-chains with the property that:
\begin{enumerate}
\item $\lim_{h \to \infty} \flat(R - P_h) = 0$;
\item $\en(P_h) \leq \E_{0}(R) + \frac{1}{h}$;
\item $\mass(P_h) \leq \mass(R) + \frac{1}{h}$.
\end{enumerate}

The due inequality will then follow in a straightforward fashion from (1) and (2). The inequality (3) is not necessary towards the proof of our result, but the possibility to produce a polyhedral flat-approximation of a rectifiable $\G$-current satisfying (2) and (3) simultaneously is an interesting byproduct of the technique.

As in Lemma \ref{lem:flat_approximation}, we adopt the notation $\pi_x$ for the affine $k$-plane $x + {\rm span}[\tau(x)]$ at any point $x \in E$ where the approximate tangent plane ${\rm Tan}(E,x)$ exists, and $S_{x,r}$ for the rectifiable $\G$-chain $\llbracket \pi_x \cap B(x,r), \tau(x) \otimes g(x) \rrbracket = g(x) \cdot \llbracket \pi_x \cap B(x,r), \tau(x), 1 \rrbracket$ for $r > 0$. Note that $\mass(S_{x,r}) = \| g(x) \| \omega_{k} r^{k}$ and $\E_{0}(S_{x,r}) = \cost(g(x)) \omega_{k} r^{k}$.

Let us also set
\[
\mu := \| g \| \Ha^{k} \trace E\,,
\]
and 
\[
\nu := \cost(g) \Ha^{k} \trace E\,.
\]
Observe that $\mu$ is a positive Radon measure in $\R^{n}$ with $\mu(\R^{n}) = \mass(R) < \infty$, and that $\nu$ is finite if and only if $\E_{0}(R) < \infty$. From now on, we will assume the validity of the latter condition, since the representation formula is evidently true if $\E_0(R)=\infty$.

Fix $\e > 0$. We make the following

{\bf{Claim}:} There exists a finite family of mutually disjoint balls $\{B_i\}_{i=1}^N$ with $B_i:=B(x_i,r_i) \subset K$ being the ball with center $x_{i} \in E$ and radius $r_{i} > 0$, such that the following properties hold:
\begin{itemize}
\item[$(i)$] $r_{i} \leq \e \qquad \forall \, i=1,\dots,N \qquad \mbox{and} \qquad \mu(\R^n \setminus (\cup_{i=1}^N  B_i))\leq \e\,;$
\item[$(ii)$] if we denote $R_{i} := R \trace B_{i}$ and $S_{i} := S_{x_{i}, r_{i}}$, then $$ \flat(R_{i} - S_{i}) \leq \e \mu(B_{i})\,; $$
\item[$(iii)$] $|\mu(B_i)- \| g(x_i) \| \omega_k r_i^k|\leq \e \mu(B_i), \qquad \forall \, i=1,\dots, N\,;$
\item[$(iv)$] if $\E_{0}(R)<\infty$, then
$$ \cost(g(x_i)) \omega_k r_i^k\leq (1+\e)\nu(B_i), \qquad \forall \, i=1,\dots, N\,.$$
\end{itemize} 

Let us assume the claim for the moment, and show how to conclude the proof of the theorem. From point $(iii)$ we deduce that
\begin{equation} \label{mass_estimate}
\mass(S_{i}) \leq (1+\e) \mass(R_i)\,,
\end{equation}
whereas point $(iv)$ implies that if $\E_{0}(R) < \infty$ then
\begin{equation} \label{cost_estimate}
\E_{0}(S_i) \leq (1+\e) \E_0(R_i)\,.
\end{equation}

Furthermore, by approximating every disc $\pi_{x_i} \cap B_{i}$ with simplexes we can conclude that there exist chains $P_{i} \in \Po_{k}^{\G}(K)$ supported on $\pi_{x_i} \cap B_{i}$ such that
\begin{equation} \label{disc_polyhedral_approx}
\flat(S_{i} - P_{i}) \leq \e \mu(B_{i}), \qquad \mass(P_{i}) \leq \mass(S_i) \quad \mbox{and} \quad \en(P_i) \leq \E_{0}(S_i)\,.
\end{equation}

Set $P := \sum_{i=1}^{N} P_{i}$. Since the balls $B_{i}$ are mutually disjoint, we have that
\begin{equation}
\en(P) = \sum_{i=1}^{N} \en(P_{i}) \leq \sum_{i=1}^{N} \E_{0}(S_i) \overset{\eqref{cost_estimate}}{\leq} (1+\e) \sum_{i=1}^{N} \E_{0}(R_i) = (1+\e) \E_{0}(R)\,, 
\end{equation}

and also that
\begin{equation}
\mass(P) = \sum_{i=1}^{N} \mass(P_{i}) \leq \sum_{i=1}^{N} \mass(S_{i}) \overset{\eqref{mass_estimate}}{\leq} (1 + \e) \sum_{i=1}^{N} \mass(R_{i}) = (1+\e) \mass(R)\,.
\end{equation}

Furthermore, we can estimate
\begin{eqnarray*}
\flat(P-R) &\leq & \sum_{i=1}^{N} \flat(P_{i} - R_{i}) + \mass\left(R \trace (\R^{n} \setminus \bigcup_{i=1}^{N} B_i)\right) \\
&\leq & \e + \sum_{i=1}^{N} \left( \flat(P_{i} - S_{i}) + \flat(S_i - R_i) \right) \\
&\overset{(ii),\eqref{disc_polyhedral_approx}}{\leq} & \e + 2\e \sum_{i=1}^{N} \mu(B_i) = \e (1+2\mass(R))\,. 
\end{eqnarray*}

This completes the proof of the theorem, provided that we show how to obtain the claim. In order to do this, let us consider the set $F$ of all points $x \in E$ such that $g(x) \neq 0$ and the following conditions are both satisfied:
\begin{itemize}
\item[$(a)$] it holds
\[
\lim_{r\to 0^+} \frac{\flat(R \trace B(x,r) - S_{x,r})}{\mass(R \trace B(x,r))} = 0 \,;
\]
\item[$(b)$] setting $\eta_{x,r}(y) := \frac{y-x}{r}$, we have that for $r \downarrow 0$ the following holds true:
\begin{align*}
& \mu_{x,r} := r^{-k} (\eta_{x,r})_{\sharp}(\mu \trace B(x,r)) \wto \| g(x) \| \Ha^{k} \trace (\pi_{x} \cap B_{1}(0))\,, \\
& \nu_{x,r} := r^{-k} (\eta_{x,r})_{\sharp}(\nu \trace B(x,r)) \wto \cost(g(x)) \Ha^{k}\trace (\pi_{x} \cap B_{1}(0))\,,
\end{align*}
\end{itemize}

where the weak\textsuperscript{*} convergence is in the sense of measures. Note that $\mu(E \setminus F) = 0$: indeed, condition $(a)$ holds true $\mu$-a.e. by Lemma \ref{lem:flat_approximation}; condition $(b)$ holds true $\mu$-a.e. by \cite[Theorem 4.8]{Camillo_book}, as both $\mu$ and $\nu$ are $k$-rectifiable Radon measures (also observe that $\mu \Lt \nu$ because of the properties of the cost function).

Now, for every $x \in F$ there exists a radius $0 < r(x) < \e$ such that
\[
\abs{\mu_{x,r}(B(0,1)) - \|g(x)\| \omega_{k}} \leq \frac{\e}{2} \|g(x)\| \omega_k \quad \mbox{for a.e. $r < r(x)$}\,,
\]

or equivalently
\begin{equation} \label{est_proj_mass}
\abs{\mu(B(x,r)) - \|g(x)\| \omega_{k} r^{k}} \leq \frac{\e}{2} \|g(x)\| \omega_{k} r^{k} \quad \mbox{for a.e. $r < r(x)$}\,.
\end{equation}

In particular, this implies that 
\begin{equation} \label{proj_mass_vs_mass}
\left( 1 - \frac{\e}{2} \right) \|g(x)\| \omega_{k} r^{k} \leq \mu(B(x,r)) \quad \mbox{for a.e. $r < r(x)$}\,,
\end{equation}

and thus, plugging \eqref{proj_mass_vs_mass} into \eqref{est_proj_mass}, we get that
\begin{equation} \label{est_mass}
\abs{\mu(B(x,r)) - \|g(x)\| \omega_{k} r^{k}} \leq \frac{\e}{2-\e} \mu(B(x,r)) \leq \e \mu(B(x,r)) \quad \mbox{for every $x \in F$, for a.e. $r < r(x)$}\,.
\end{equation}

Analogous computations show that
\begin{equation} \label{est_cost}
\abs{\nu(B(x,r)) - \cost(g(x)) \omega_{k} r^{k}} \leq \frac{\e}{2-\e} \nu(B(x,r)) \leq \e \nu(B(x,r)) \quad \mbox{for every $x \in F$, for a.e. $r < r(x)$}\,.
\end{equation}

The claim is then a simple consequence of the Vitali-Besicovitch covering theorem.
\end{proof}

\bibliographystyle{plain}
\bibliography{MMST}

\end{document}